\documentclass[a4paper, 12pt]{amsart} 
\usepackage{amssymb,amsmath,amsfonts,latexsym,color,graphicx,stmaryrd,dsfont}
\usepackage{hyperref} 
\hypersetup{pdfborder=0 0 0, 
	    colorlinks=true,
	    citecolor=blue,
	    linkcolor=blue,
	    urlcolor=red,
	    pdfauthor={Johannes Sjoestrand, Martin Vogel}
	   }
\usepackage[font=small,labelfont=bf]{caption}
\numberwithin{equation}{section}
\newcommand{\C}{\mathbf{C}}

\newcommand{\dist}{\mathrm{dist}}
\newcommand{\e}{\mathrm{e}}

\newcommand{\mO}{\mathcal{O}}

\newtheorem{dref}{Definition}[section] \newtheorem{lemma}[dref]{Lemma}
\newtheorem{theo}[dref]{Theorem} \newtheorem{prop}[dref]{Proposition}
\newtheorem{remark}[dref]{Remark}

\title[Bi-diagonal matrices and random perturbations]{Interior eigenvalue density of large bi-diagonal matrices subject to random perturbations}
\author{Johannes Sj\"ostrand}
\address[Johannes Sj\"ostrand]{IMB, 
  Universit\'e de Bourgogne Franche-Comt\'e, 
  UMR 5584 du CNRS, 
  9, avenue Alain Savary - BP 47870 FR-21078 Dijon Cedex.}
\email{johannes.sjostrand@u-bourgogne.fr}
\author{Martin Vogel}
\address[Martin Vogel]{D\'epartement de Math\'ematiques - UMR 8628 CNRS, B\^atiment 440, Universit\'e Paris-Sud, 
15 Rue du Doyen Georges Poitou, F-91405 Orsay Cedex.}
\email{martin.vogel@math.u-psud.fr}
 \date{}
\dedicatory{Dedicated to Professor Takahiro Kawai and Professor Hikosaburo Komatsu}
\begin{document}
\begin{abstract}
 We study the spectrum of large a bi-diagonal Toeplitz matrix subject to 
 a Gaussian random perturbation with a small coupling constant. 
 We obtain a precise asymptotic description of the average 
 density of eigenvalues in the interior of the convex hull of the range 
 symbol. 
  \vskip.5cm
  \par\noindent \textsc{R{\'e}sum{\'e}.} 
 Nous \'etudions le spectre d'une grande matrice de Toeplitz 
 soumise \`a une perturbation gaussienne avec petite constante 
 de couplage. Nous obtenons une description asymptotique pr\'ecise 
 de la densit\'e moyenne des valeurs propres \`a l'int\'erieur l'enveloppe 
 convexe de l'image du symbole. 
\end{abstract}
\maketitle
\setcounter{tocdepth}{1}
%
%
\section{Introduction and main result}\label{int}
\setcounter{equation}{0}
It is well known that the spectrum of non-normal operators can be extremely 
unstable even under tiny perturbations, see e.g. \cite{TrEm05,Da07}. It is 
therefore a natural question to study the spectra of such operators subject 
to small random perturbations. Recently, there has been a mounting interest 
in the spectral properties of elliptic non-normal (pseudo-)differential operators 
with small random perturbations, see for example \cite{BM,Ha06b,HaSj08,SjAX1002,Vo14,ZwChrist10}. 
An interesting, perhaps surprising, result is that by adding a small random 
perturbation, we can obtain a probabilistic Weyl law for the eigenvalues 
for a large class of such operators. 
\par
Another important example is the case of non-normal Toeplitz matrices, 
since they can arise for example in models non-hermitian quantum 
mechanics, see e.g. \cite{GoKh00,HaNe96}. The authors' interest in this 
case, however, is motivated by the aspect of spectral instability. 
%
\\
\par
The goal of this work is to study the spectrum of random perturbations of the
following bidiagonal $N\times N$  Toeplitz matrix:
\begin{equation}\label{int.1}
P=\begin{pmatrix} 0 &a &0 &.. &.. &0\\
b &0 &a &.. &..&0\\
0 &b &0 &.. &..&0\\
.. &.. &.. &..&..&..\\
0 & ..&.. &..&0 &a\\
0 &0 &.. &.. &b &0 \end{pmatrix}.
\end{equation}
Here $a,\, b\in {\bf C}\setminus \{ 0 \}$ and $N\gg 1$. Identifying ${\bf C}^N$
with $\ell^2([1,N])$, $[1,N]=\{ 1,2,..,N\}$ and also with $\ell^2_{[1,N]}({\bf
  Z})$ (the space of all $u\in \ell^2({\bf Z})$ with support in
$[1,N]$), we have:
\begin{equation}\label{int.3}
P=1_{[1,N]}(a\tau _{-1}+b\tau _1)1_{[1,N]}=1_{[1,N]}(a\e^{iD_x}+b\e^{-iD_x})1_{[1,N]},
\end{equation}
where $\tau _ku(j)=u(j-k)$ denotes translation by $k$, and 
	\begin{equation*}
		(a\e^{iD_x}+b\e^{-iD_x})u(n)= \frac{1}{2\pi} 
		\int_{\mathbf{R}/2\pi\mathbf{Z}}\mathrm{e}^{in\xi}p(\xi)\widehat{u}(\xi)d\xi,
		\quad u \in \ell^2(\mathbf{Z}),
	\end{equation*}
where $\widehat{u}$ denotes the Fourier transformation of $u$ and 
$p(\xi)$ is the symbol of $P$, given by 
\begin{equation}\label{int.5}
p(\xi )=a\e^{i\xi }+b\e^{-i\xi }.
\end{equation}
Assume, to fix the ideas, that $|b|\le |a|$. Then
$p({\bf R})$ is equal to the ellipse, $E_1$, centred at 0 with major
semi-axis of length $(|a|+|b|)$ pointing in the direction $e^{i(\alpha
+\beta )/2}$, where $\alpha = \mathrm{arg}(a)$, $\beta=\mathrm{arg}(b)$, 
and minor semi-axis of length $|a|-|b|$. The focal points
of $E_1$ are
\begin{equation}\label{rasy.2.5}
\pm 2\sqrt{ab}=\pm e^{i\frac{\alpha +\beta }{2}} 2\sqrt{|a||b|}.
\end{equation}
In a previous work \cite{SjVo15b} the authors have shown that
the numerical range of $P$ is contained in the
convex hull of the ellipse $E_1$ described above and the 
eigenvalues of $P$ are given by 
\begin{equation}\label{spnp.12}
  z=z(\nu )=2\sqrt{ab}\cos 
  \left( \frac{\pi \nu }{N+1} \right), 
  \quad \nu =1,\dots, N.
\end{equation}
This result is also illustrated in Figure \ref{fig1}. In this work, we consider 
the following random perturbation of $P$
\begin{equation}\label{int.5a}
 P_{\delta} := P + \delta Q_{\omega}, 
 \quad
 Q_{\omega}=(q_{j,k}(\omega))_{1\leq j,k\leq N}, 
\end{equation}
where $0\leq\delta\ll 1 $, possibly depending on $N$, 
and $q_{j,k}(\omega)$ are independent and 
identically distributed complex Gaussian random variables, 
following the complex Gaussian law $\mathcal{N}_{\C}(0,1)$. 
\begin{figure}[h]
\centering
\includegraphics[scale=0.6]{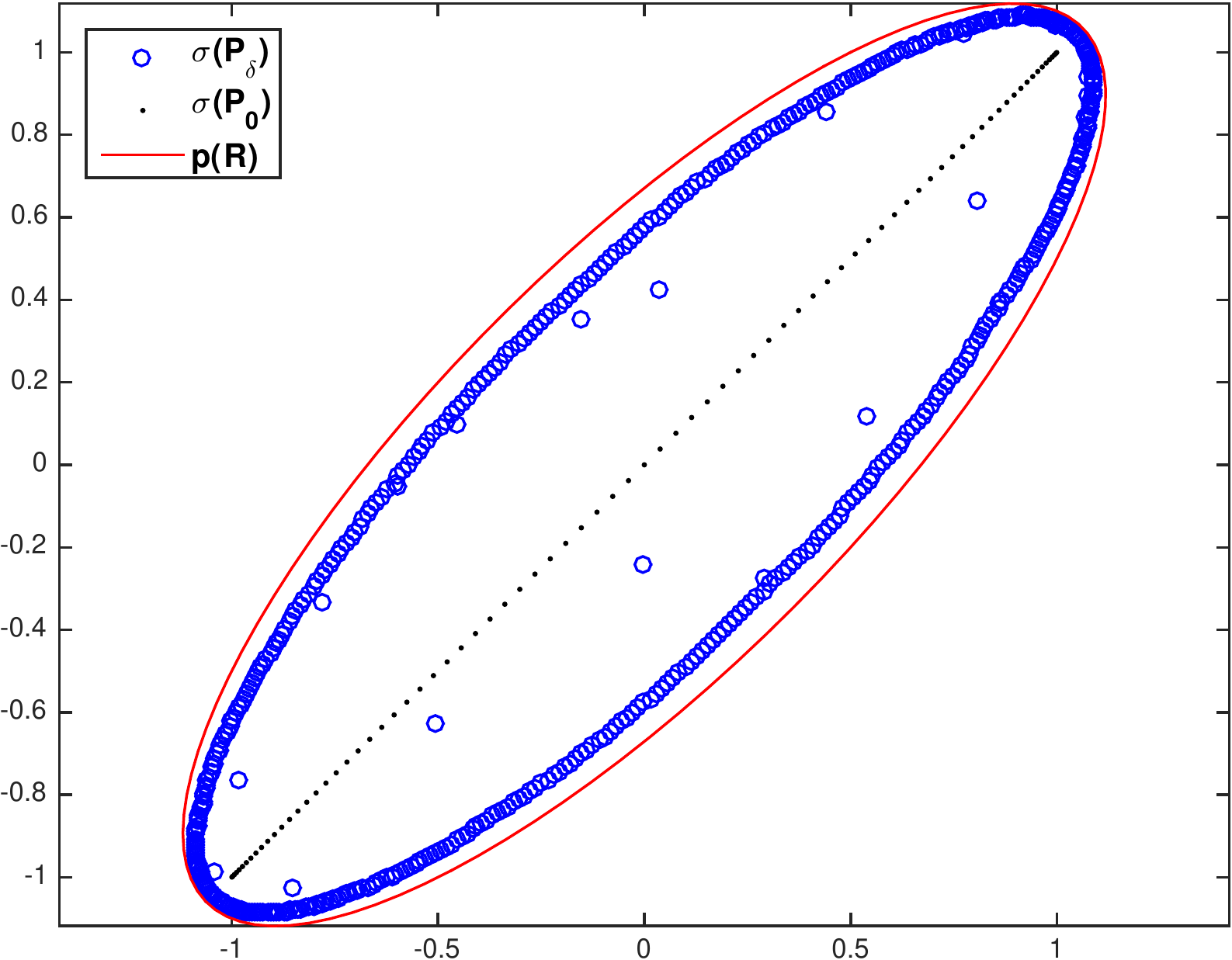}
\caption{The black dots along the focal segment show the 
spectrum (obtained using MATLAB) of the unperturbed operator $P$ with dimension $N=501$, 
$a=0.5$, $b=i$ and $\delta=10^{-12}$. The blue cirlces show the spectrum 
of the perturbed operator \eqref{int.5a}, and the red ellipse is the image of 
the symbol $p$.}
\label{fig1}
\end{figure}
In \cite{SjVo15b}, the authors proved that when the coupling constant $\delta$ is 
bounded from above and from below by sufficiently negative powers of $N$, 
then most eigenvalues of $P_{\delta}$, \eqref{int.5a}, 
are close to the ellipse $p(\mathbf{R})$ and follow a Weyl law, with probability 
close to one, as the dimension $N$ gets large (cf. Figure \ref{fig1}).
\par
The methods used in \cite{SjVo15b} are essentially based on probabilistic 
subharmonic estimates of $\ln |\det(P_{\delta}-z)|$ and complex analysis, 
using in particular a counting theorem of \cite{Sj09b} (see also \cite{Ha06,HaSj08}). 
However, this approach is not fine to enough give a detailed description of the 
exceptional eigenvalues seen inside the ellipse in Figure \ref{fig1} and 
we only obtain a logarithmic upper bound on the number of eigenvalues in this 
region. To gain more information about these eigenvalues, we study the random measure 
	\begin{equation}
		\Xi  :=\sum_{z\in\sigma(P_{\delta})}\delta_z,
	\end{equation}
where the eigenvalues are counted with multiplicity. In particular we are 
interested in studying the first intensity measure of $\Xi$, which is the 
positive measure $\nu$ defined by 
	\begin{equation}
	\mathds{E} \left[ \Xi(\varphi)\right] = \int \varphi(z)\nu(dz), 
	\end{equation}
where $\varphi$ is a test function of class $\mathcal{C}_0$. The measure 
$\nu$ contains information about the average density of eigenvalues, 
and we will show in Theorem \ref{thm1} below, that it admits a continuous 
density with respect to the Lebesgue measure on $\C$, up to a small error 
in the large $N$ limit.
\par 
This approach 
is more classical in the theory of random polynomials (cf. \cite{SZ03,BSZ00}) 
and random Gaussian analytic functions (cf. \cite{HoKrPeVi09,So00}). We 
follow in particular the approach developed in \cite{Vo14}, which was 
therein used to describe the average density of eigenvalues of a class of semiclassical differential 
operators subject to small random perturbations.
\\
\par
The main result of this paper describes the average density of eigenvalues 
in the interior of confocal ellipses.  Let $p_{a,b}=p$ as in \eqref{int.5}. 
For any $r>0$ we define $\Sigma_r$ to be the convex hull of 
$p_{ra,r^{-1}b}(\mathbf{R})$. We will see in Section \ref{sizz} that 
$p_{ra,r^{-1}b}(\mathbf{R})$, for $ (|b|/|a|)^{1/2} \leq r < +\infty$, are confocal 
ellipses and that they are in the interior of $\Sigma_{r_0}$, for every $r_0>r$. 
Moreover that $p_{ra,r^{-1}b}(\mathbf{R})$, with $r= (|b|/|a|)^{1/2}$, is the 
focal segment. 
\par
We prove the following result.
\begin{theo}\label{thm1}
Let $P_{\delta}$ be as in \eqref{int.5a} and let $p_{a,b}=p$ as in \eqref{int.5}. 
Let $C\gg 1$ be arbitrary, but fixed (and not necessarily the same in the sequel). 
Let $r_1=|b/a|^{1/2} +1/C$, let $\e^{-N/C} \leq \delta \ll 1$, $N\gg 1$ 
and let $r_0>0$ belong to the parameter range 
\begin{equation}\label{eq1.1}
\begin{split}
& \frac{1}{C} \leq r_0 \leq 1- \frac{1}{N}, \\ 
& \frac{Nr_0^{N-1}}{\delta}(1-r_0)^2 +\delta N^3 \ll  1,
\end{split}
\end{equation}
so that $\delta N^3 \ll 1$. For $r>0$, let $\Sigma_r$ be the convex hull of 
$p_{ra,r^{-1}b}(\mathbf{R})$. Then, for all 
$\varphi\in\mathcal{C}_0(\mathring{\Sigma}_{(r_0-1/N)}\backslash{\Sigma}_{r_1} )$, 
\begin{equation}\label{eq1}
 \mathds{E} \left[  \sum_{\lambda\in\sigma(P_{\delta})}\varphi(z)
 \right] =
  \int \varphi(z)\xi(z) L(dz)
  + \langle \mu_N, \varphi \rangle,
\end{equation}
for some $C\gg 1$. Here, the density $\xi$ is a continuous function 
satisfying, 
\begin{equation}\label{eq2}
	\begin{split}
	&\xi(z) = \frac{2}{\pi}\partial_{z}\partial_{\bar{z}}\ln K(z) 
	\left(1 +\mathcal{O}\!\left(\frac{N |\zeta_-|^{N-1}}{\delta }(1-|\zeta_-|)^2
 	+ \delta N^3
	\right)\right), \\ 
	&K(z)= \sum_{k=0}^{\infty}
	 \left|\frac{\zeta_-^{k+1} - \zeta_+^{k+1}}{a(\zeta_- - \zeta_+)}\right|^2,
	\end{split}
\end{equation}
where $\zeta_{\pm}(z)$ are the two solutions of the equation $p_{a,b}(\zeta)=z$ 
for $z\in\Sigma_1 \backslash [-2\sqrt{ab},2\sqrt{ab}]$, chosen such that 
$|\zeta_-| \geq |\zeta_+|$. $\partial_{z}\partial_{\bar{z}}\ln K(z) $ 
is smooth and strictly positive. 
\par
Furthermore, $\mu_N$ is a Radon measure of total mass $\leq N\mathrm{e}^{-N^2}$, 
i.e. $ |\langle \mu_N, \varphi \rangle | \leq N\mathrm{e}^{-N^2}\|\varphi\|_{\infty}$. 
\end{theo}
Let us give some remarks on this result. We will show in Section \ref{sizz} that 
for $p(\zeta_{\pm})=z\in\mathring{\Sigma}_1\backslash [-2\sqrt{ab},2\sqrt{ab}]$ 
we have that $|\zeta_+| < |b/a|^{1/2} < |\zeta_-| <1$. In fact we have that 
$|\zeta_-| \leq r_0$ when $z\in\Sigma_{r_0}\backslash [-2\sqrt{ab},2\sqrt{ab}]$. 
\\
\par
Secondly, for $r_0$ satisfying the first condition in 
\eqref{eq1.1}, the function $[0,r_0]\ni r \mapsto r^{N-1}(1-r)^2$ is increasing. Hence, 
the error term in \eqref{eq2} is small, since it is dominated by the term in the second 
line of \eqref{eq1.1}. More precisely, it satisfies for $|\zeta_-|\leq r_0$
	\begin{equation*}
	\frac{N |\zeta_-|^{N-1}}{\delta }(1-|\zeta_-|)^2+ \delta N^3
	\leq
	\frac{Nr_0^{N-1}}{\delta}(1-r_0)^2 +\delta N^3.
	\end{equation*}
Theorem \ref{thm1} shows that in the interior of the ellipse $p(\mathbf{R})$ (see 
Figure \ref{fig1}) there is a non-vanishing continuous density of eigenvalues 
whose leading term is independent of the dimension $N$ and depends only 
the symbol $p$. 
\par
Furthermore, we note that the leading term of the density $\xi$ is related to 
the Edelman-Kostlan formula (see for example \cite{HoKrPeVi09}) 
for the average density of the zeros of a Gaussian analytic function $g(z)$, 
in the sense of \cite{HoKrPeVi09}, with covariance kernel $K(z)$, i.e.
	\begin{equation*}
		\mathds{E}[g(z)\overline{g(z)}] = K(z).
	\end{equation*}
The above theorem, together with the result of \cite{SjVo15b}, is a generalisation 
of the work done in the case where the unperturbed operator $P$ is given by 
a large Jordan block, i.e. the case where $a=1$, $b=0$. This has already 
been subject to intense study : 
M.~Hager and E.B.~Davies \cite{DaHa09} showed that with a sufficiently small 
coupling constant 
most eigenvalues of $P_{\delta}$ can be found near a circle, with probability close to $1$, as 
the dimension of the matrix $N$ gets large. This result has been refined by one of the 
authors in \cite{Sj15}, showing that, with probability close to $1$, most eigenvalues 
follow an angular Weyl law. Furthermore, M.~Hager and E.B.~Davies \cite{DaHa09} give a probabilistic upper bound of order $\log N$ for the number of eigenvalues in the interior of a circle. 
\par
A recent result by A.~Guionnet, P.~Matched Wood and 
O.~Zeitouni \cite{GuMaZe14} implies that when the coupling constant is 
bounded from above and from below by (different) sufficiently negative powers of $N$, then 
the normalized counting measure of eigenvalues of the randomly perturbed Jordan block  converges weakly in probability to the uniform measure on $S^1$ as the dimension of the 
matrix gets large. 
\par
In \cite{SjVo15}, the authors show that in the case where $P$ is given by 
a Jordan block matrix, the leading term of the average density of eigenvalues 
is given by the density of the hyperbolic volume on the unit disk. 
\par
A similar result has been obtained by C.~Bordenave and 
M.~Capitaine in \cite{BoCa16}, where they allow for a more general class 
of random matrices, however, with slower decay of the coupling 
constant, as $N\gg 1$. In particular they show that the point process 
$\Xi$ converges weakly inside some disc, in the limit $N\to\infty$, to 
the point process given by 
the zeros of a certain Gaussian analytic function (in the sense 
of \cite{HoKrPeVi09}) on the Poincar\'e disc.
%
\\[2ex]
\textbf{Acknowledgements.} M.~Vogel was supported by the project  
GeRaSic ANR-13-BS01-0007-01.%
\section{Image of the symbol $p$}\label{sizz}
\setcounter{equation}{0}
It will be important to understand 
the solutions of the characteristic equation $p(\xi)=z$. The discussion 
that follows has been taken from \cite{SjVo15b} and is presented here  
for the reader's convenience. 
\\
\par
We recall that we have assumed for simplicity that
$|a|\ge |b|$. The case $|a|=|b|$ will be obtained as a limiting case
of the one when $|a|>|b|$, that we consider now. We write the symbol 
$p$ \eqref{int.5} in the form
$$
f_{a,b}(\zeta )=a\zeta +b/\zeta, \quad \zeta=\mathrm{e}^{i\xi}, 
$$
and observe that when $r>0$
$$
f_{a,b}(\partial D(0,r))=f_{ar,b/r}(\partial D(0,1))
$$
which gives a family of confocal ellipses $E_r$. The length of the
major semi-axis of $E_r$ is equal to $|a|r+|b|/r=:g(r)$. $E_{r_1}$ is contained in the bounded domain which has $E_{r_2}$
as its boundary, precisely when $g(r_1)\le g(r_2)$. The function $g$
has a unique minimum at $r=r_\mathrm{min}=(|b|/|a|)^{1/2}$. $g$ is
strictly decreasing on $]0,r_\mathrm{min}]$ and strictly increasing on
$[r_\mathrm{min},+\infty [ $. It tends to $+\infty $ when $r\to 0$ and
when $r\to +\infty $. We have
$g_\mathrm{min}=g(r_\mathrm{min})=2(|a||b|)^{1/2}$ so
$E_{r_\mathrm{min}}$ is just the segment between the two focal points,
common to all the $E_r$. For $r\ne r_\mathrm{min}$, the map $\partial
D(0,r)\to E_r$ is a diffeomorphism. Let $r_1$ be the unique value in
$]0,1[$ for which $g(r_1)=|a|+|b|=g(1)$. We get the following result:
\begin{prop}\label{sizz1} Let $|b|<|a|$. 
\begin{itemize}
\item
When $z$ is strictly inside the ellipse $E_1$ described above, 
then both solutions of $f_{a,b}(\zeta )=z$ belong to $D(0,1)$.
\item When $z$ is on the ellipse, one solution is on $S^1$ and the
  other belongs to $D(0,1)$.
\item When $z$ is in the exterior region to the ellipse, one solution
  fulfils $|\zeta |>1$ and the other satisfies $|\zeta |<1$. 
\end{itemize}
\end{prop}
\par In the case $|a|=|b|$, $E_1$ is just the segment between the two
focal points. In this case $r_\mathrm{min}=1$ and we get:
\begin{prop}\label{sizz2}
Assume that $|a|=|b|$. \begin{itemize} 
\item If $z\in E_1$ then both solutions of
$f_{a,b}(\zeta )=z$ belong to $S^1$.
\item If $z$ is outside $E_1$, one solution is in $D(0,1)$ and the
  other is in the complement of $\overline{D(0,1)}$.
\end{itemize}
\end{prop}
\begin{remark}\label{remh1}
Assuming that $0<|b|\leq |a|$, we observe that for 
$z\in\C$ the two solutions, say $\zeta_{\pm}$ of $f_{a,b}(\zeta)=z$ 
are solutions of the equation 
\begin{equation}
	\zeta^2 -\frac{z}{a}\zeta + \frac{b}{a}=0,
\end{equation}
and they satisfy the relations 
\begin{equation}\label{algrel}
	 \zeta_+\zeta_- = \frac{b}{a}, 
	 \quad 
	 \zeta_+ + \zeta_- = -\frac{z}{a}.
\end{equation}
Furthermore, we can fix a branch of the square root such that 
$\zeta_+(z)$ and $\zeta_-(z)$ are holomorphic functions of $z$ 
in $\C\backslash [-2\sqrt{ab},2\sqrt{ab}]$.
\end{remark}
Throughout this text, we will work with the convention that 
\begin{equation}\label{signcon}
	|\zeta_+| \leq |\zeta_-|
\end{equation}
which in particular yields by the above discussion that when 
$z$ is inside $E_r$, for $r\in [r_{\mathrm{min}}, +\infty[$, then 
\begin{equation}\label{signcon2}
	0< |\zeta_+| \leq \sqrt{|b/a|} \leq |\zeta_-| \leq r.
\end{equation}
\section{Preparations for the density of eigenvalues in the interior}
In this section we are interested in the density of eigenvalues in the 
interior of the ellipse $p_{a,b}(\mathbf{R})$, where $p_{a,b}=p$ denotes the 
principal symbol of the unperturbed operator $P$, cf. \eqref{int.3}, \eqref{int.5}. 
We study the first moment of linear statistics of the point process 
given by the eigenvalues of $P_{\delta}$, see \eqref{int.5a}, i.e.
\begin{equation}\label{eq10.60}
 I_{\varphi} = \mathds{E} \left[ \sum_{\lambda\in\sigma(P_{\delta})}\varphi(z)
 \right], \quad 
 \varphi \in \mathcal{C}_0(\Omega),
\end{equation}
where $\Omega$ is some open subset in the interior of 
$\mathrm{conv}(p_{a,b}(\mathbf{R}))\backslash[-2\sqrt{ab},2\sqrt{ab}]$, 
where $\mathrm{conv}(\cdot)$ denotes the convex 
hull of a set.
\par
W.~Bordeaux-Montrieux \cite{BM} noted that the Markov inequality implies 
that if $C_1>0$ is large enough, then 
for the Hilbert-Schmidt norm of $Q_{\omega}$ (as in \eqref{int.5a}), 
\begin{equation}\label{grpp.0b}
\mathds{P}\left[ 
\Vert Q_{\omega}\Vert_\mathrm{HS}\le C_1N
\right] \ge 1-e^{-N^2}.
\end{equation}
Since the number of eigenvalues of $P_{\delta}$ in the support 
of $\varphi$ is bounded from above by $N$, it follows from \eqref{grpp.0b} that 
\begin{equation}\label{eq10.61}
 \begin{split}
 &I_{\varphi} = \mathds{E} \left[ \mathds{1}_{B_{\C^{N^2}}(0,C_1N)}(Q) \sum_{\lambda\in\sigma(P_{\delta})}\varphi(z)
 \right] + \langle \mu_N, \varphi \rangle, \\
& |\langle \mu_N, \varphi \rangle | \leq N\mathrm{e}^{-N^2}\|\varphi\|_{\infty}.
 \end{split}
\end{equation}
Here, we identify the random matrix $Q_{\omega}$ (cf \eqref{int.5a}) with 
a random vector $Q\in\C^{N^2}$. Furthermore, $\mu_N$ is a Radon measure 
of total mass $\leq N\mathrm{e}^{-N^2}$. 
\par
After the reduction to \ref{eq10.61}, it is sufficient to work with the assumption 
that the random vector $Q$ is restricted to a ball of radius $C_1N$, i.e.
	\begin{equation}
		\|Q\|_2\leq C_1 N.
	\end{equation}
Note that this assumption is equivalent, to the assumption that the Hilbert-Schmidt 
norm of the random matrix $Q_{\omega}$ is bounded, more precisely that 
	\begin{equation}\label{eqa1}
		\|Q\|_{HS}\leq C_1 N.
	\end{equation}
Next, we define for $r>0$ 
\begin{equation}\label{eq10.1}
	\Sigma_{r}:= \mathrm{conv}(p_{ar,br^{-1}}(\mathbf{R})).
\end{equation}
We let
\begin{equation}\label{eq10.1.5}
	\Omega\Subset\mathring{\Sigma}_1\backslash 
	[-2\sqrt{ab},2\sqrt{ab}], 
\end{equation}
be open, relatively compact and connected. It 
may depend on $N$ (to be specified later on) but will 
avoid a fixed neighbourhood of the focal segment. Moreover, 
let $W=B(0,C_1N)$ for $C_1>0$ large enough such that 
\eqref{grpp.0b} holds. By remark \ref{remh1} we see that by excluding 
the focal segment in \eqref{eq10.1.5} we have that $\zeta_{\pm}(z)$, 
the solutions to the characteristic equation, given by the symbol \eqref{int.5}, 
\begin{equation*}
	a\zeta + b\zeta^{-1} = z,
\end{equation*}
are holomorphic functions of $z$,. 
\\
\par
In the following we write for $\mu\in\mathbf{N}$
\begin{equation}\label{eq10.2.1}
	F_{\mu+1}(t) = 1 + t + \dots + t^{\mu}
	, \quad 0\leq t \leq 1.
\end{equation}
As in \cite{SjVo15b}, we work under the hypothesis that 
\begin{equation}\label{grpp.1}
 \delta N F_N(|\zeta_-|) \ll 1.
\end{equation}
Notice that this is fulfilled for all $z$ inside $E_1=p(\mathbf{R})$, if we make 
the even stronger assumption 
\begin{equation}\label{grpp.2}
 \delta N^2 \ll 1.
\end{equation}
(Recall that $N\gg 1$). 
We have shown in \cite{SjVo15b} that assuming \eqref{grpp.1}, \eqref{eqa1}
we can identify the eigenvalues of 
$P_{\delta}$ in $\Omega$ with the zeros of $g(z,Q)$, a holomorphic function 
on $\Omega\times W$. Note that since there are at most $N$ eigenvalues, 
we have for every $Q\in W$ that $g(\cdot,Q)\not\equiv 0$. Furthermore, 
see \cite[Formula (8.18)]{SjVo15b}, $g$ is given by
\begin{equation}\label{eq10.2}
g(z,Q) = g_0(z) - \delta (Q|\overline{Z})+ T(z,Q;\delta,N),
\end{equation}
where $Z$ is given by 
\begin{equation}\label{grpp.14}
\begin{split}
 Z &= \left( 
 \frac{\zeta_+^{N+1-j} - \zeta_-^{N+1-j}}{a(\zeta_+ - \zeta_-)} 
 \frac{\zeta_+^{k} - \zeta_-^{k}}{a(\zeta_+ - \zeta_-)} 
 \right)_{1\leq j,k\leq N} \\
 &=\left( a^{-2}
 F_{N+1-j}(\zeta_+/\zeta_-)F_{k}(\zeta_+/\zeta_-) \zeta_-^{N-j+k-1}
 \right)_{1\leq j,k\leq N},
 \end{split}
\end{equation}
and
\begin{equation}\label{eq10.3}
 g_0(z) = \frac{\zeta_-^{N+1} - \zeta_+^{N+1}}{a(\zeta_- - \zeta_+)}
 	    = \frac{\zeta_-^N}{a} F_{N+1}(\zeta_+/\zeta_-).
\end{equation}
Moreover, 
\begin{equation}\label{eq10.3.5}
|T(z,Q)|=|T(z,q;\delta,N)|=\mathcal{O}(1)(\delta N F_N(|\zeta_-|^2))^2.
\end{equation}
We will frequently write $|\cdot |$ for the Hilbert-Schmidt norm and, 
until further notice, we write $F_{\mu}=F_{\mu}(\zeta_+/\zeta_-)$. 
By \eqref{grpp.14}, we get that 
\begin{equation}\label{eq10.3.4}
 |Z| =|a|^{-2} \left(\sum_{j,k=1}^N|\zeta_-|^{2(N-j+k-1)}|F_{N+1-j}|^2|F_{k}|^2\right)^{\frac{1}{2}}
 = |a|^{-2}\sum_{\mu =0 }^{N-1}|\zeta_-|^{2\mu}|F_{\mu+1}|^2.
\end{equation}
For $z\in\Omega$ we have $|\zeta_+|/|\zeta_-|\leq C < 1$ and hence 
$|F_k(\zeta_+/\zeta_-)|\asymp 1$. If we also assume $z\in\Sigma_{r_0}$, 
$0<r_0\leq 1-1/N$, then 
\begin{equation}\label{eq10.3.6}
 |Z| \asymp F_N(|\zeta_-|^2) \asymp \frac{1}{1-|\zeta_-|^2} \asymp \frac{1}{1-|\zeta_-|}, 
\end{equation}
where we used as well that $\sqrt{|b/a|} \leq |\zeta_-| \leq 1- 1/N$ 
(see \eqref{signcon2},\eqref{eq10.8}, \eqref{eq10.9}), and that 
	\begin{equation}
		F_N(|\zeta_-|^2) = \frac{1}{1-|\zeta_-|^2}(1 - |\zeta_-|^{2(N+1)})
		\asymp \frac{1}{1-|\zeta_-|^2}.
	\end{equation}
Recall that $\Omega$ in \eqref{eq10.1.5} avoids a fixed neighborhood of 
the focal segment of the ellipse $E_1=p(\mathbf{R})$. More precisely, in 
view of the discussion in Section \ref{sizz}, we assume that 
\begin{equation}\label{eq10.4}
	\begin{cases}
	\Omega\Subset\mathring{\Sigma}_{1}\backslash \Sigma_{r_1}, \\
	r_1= \sqrt{|b/a|} + 1/C, ~ C\gg 1.
	\end{cases}
\end{equation}
%
Using \eqref{eq10.4}, it follows that the middle term in \eqref{eq10.2} 
is bounded in modulus by
\begin{equation}\label{eq10.5}
	\delta |Q| |Z| \leq \mO(1)(C_1\delta N F_N(|\zeta_-|^2))
\end{equation}
where we assumed that $|Q|\leq C_1 N$ (cf. \eqref{grpp.1}). Moreover, we 
assume that the 
first term in \eqref{eq10.2} is smaller than the bound on the middle term, i.e.
\begin{equation}\label{eq10.6}
	|g_0(z)| \ll C_1\delta N F_N(|\zeta_-|^2).
\end{equation}
Using that $|F_k(\zeta_+/\zeta_-)|\asymp 1$, we see that \eqref{eq10.6} is 
implied by the assumption 
\begin{equation}\label{eq10.7}
	|\zeta_-|^N \ll C_1\delta N F_N(|\zeta_-|^2).
\end{equation}
More precisely, we will assume that $z$ satisfying \eqref{eq10.4} is 
such that $\zeta_-(z)\in D(0,r_0)$ with 
\begin{equation}\label{eq10.8}
	|r_0|^N \ll C_1\delta N F_N(r_0^2), \quad r_0\leq 1 - \frac{1}{N}.
\end{equation}
Observe that the function $r^N/F_N(r^2)$ is strictly monotonically growing on 
the interval $[0,1-N^{-1}]$. Thus, the inequality \eqref{eq10.7} 
is preserved if we replace $r_0$ by $|\zeta_-|$, for $|\zeta_-|\leq r_0$.
\par
Combining the assumptions \eqref{eq10.4} and \eqref{eq10.7}, we get 
\begin{equation}\label{eq10.9}
	\begin{cases}
	z\in\Omega\Subset\Sigma_{r_0,r_1}:=\mathring{\Sigma}_{r_0}\backslash \Sigma_{r_1}, \\
	r_0>0 \text{ satisfies  \eqref{eq10.8}}, \\
	r_1= \sqrt{|b/a|} + 1/C, ~ C\gg 1.
	\end{cases}
\end{equation}
By \eqref{grpp.1}, we see that the bound on $T$ is much smaller than the 
upper bound on the middle term in \eqref{eq10.2}, i.e.
\begin{equation}\label{eq10.10}
	(\delta NF_{N+1}(|\zeta_-|^2))^2 \ll \delta NF_{N}(|\zeta_-|^2)
\end{equation}
Here we used as well that $F_{N+1}(|\zeta_-|^2) \asymp F_{N}(|\zeta_-|^2)$. 
From \eqref{eq10.2}, \eqref{eq10.3.5} and the Cauchy inequalities, we 
get 
\begin{equation}\label{eq10.11}
	d_Qg(z,Q) = -\delta Z\cdot dQ + \mO(\delta^2F_{N+1}^2(|\zeta_-|^2) N)
\end{equation}
where the norm of the first term is 
$\asymp \delta |Z| \asymp \delta F_N(|\zeta_-|^2) \gg 
\delta^2F_{N+1}^2(|\zeta_-|^2)  N$. Here, we used \eqref{grpp.1}, 
\eqref{eq10.3.6}. Technically, we need to apply the Cauchy inequalities 
in a ball of radius $\eta C_1N$ for some $0<\eta<0$, but we have 
room for that if we choose $C_1$ in \eqref{grpp.1} slightly larger 
to begin with. 
\par
Recall that for every $Q\in W$, $g(\cdot,Q)\not\equiv 0$. It has then been shown 
in \cite{Vo14,SjVo15}, that if 
\begin{equation*}
	g(z,Q) = 0 \Rightarrow
	d_Qg(z,Q) \neq 0
\end{equation*}
then 
\begin{equation}\label{eq10.11.1}
	\Gamma := \left\{ 
	(z,Q)\in\Omega\times W; g(z,Q)=0
	\right\}
\end{equation}
is a smooth complex hypersurface in $\Omega\times W$ and 
\begin{equation}\label{eq10.11.2}
K_{\varphi} =
\mathds{E} \left[ \mathds{1}_{B(0,C_1N)}(Q) \sum_{\lambda\in\sigma(P_{\delta})}\varphi(z)
 \right] 
 =
 \int_{\Gamma}\varphi(z)\e^{-Q^*Q}\,\frac{j^*(d\overline{Q}\wedge dQ)}{(2i)^{N^2}},
\end{equation}
where $j^*$ denotes the pull-back by the regular embedding 
$j:\Gamma \to \Omega\times W$ and 
$$
	d\overline{Q}\wedge dQ = d\overline{Q}_1\wedge dQ_1 \wedge \dots 
	d\overline{Q}_N\wedge dQ_N,
$$
which is a complex $(N^2,N^2)$-form on $\Omega\times W$. Thus, 
$(2i)^{-N^2}j^*(d\overline{Q}\wedge dQ)$ is a non-negative differential form on $\Gamma$ 
of maximal degree.
\\
\par
Next, we identify $Z(z)$ in \eqref{grpp.14} with a vector in 
$\C^{N^2}$ and write 
\begin{equation}\label{eq10.12}
	Q = Q(\alpha) =\alpha_1\overline{Z}(z) + \alpha', \quad 
	\alpha_1\in\C, ~\alpha'\in \overline{Z}(z)^{\perp}
\end{equation}
and we identify $\overline{Z}(z)^{\perp}$ unitarily with $\C^{N^2-1}$ by 
means of an orthonormal basis $e_2(z),\dots,e_{N^2}(z)$, so that 
$\alpha'=\sum_2^{N^2} \alpha_je_j(z)$. Then, we have
	\begin{equation}\label{eq10.12.1}
	Q=Q(\alpha,z) = \alpha_1\overline{Z}(z) + \sum_2^{N^2} \alpha_je_j(z)
	\end{equation}
and we identify $g(z,Q)$ with $\tilde{g}(z,\alpha)=g(z,Q(\alpha,z))$ which is 
holomorphic in $\alpha$ for every 
fixed $z$ and, by \eqref{eq10.2}, \eqref{eq10.3.5}, we have that 
\begin{equation}\label{eq10.13}
\begin{split}
&\tilde{g}(z,\alpha) = g_0(z) -\delta |Z|^2\alpha_1 + 
T\!\left(z,\alpha_1\overline{Z}(z)+ \sum_2^{N^2}\alpha_je_j(z)\right) \\ 
&\partial_{\alpha_1}\tilde{g}(z,\alpha) = -\delta |Z|^2 + \mO(\delta^2F_{N+1}^3 N).
\end{split}
\end{equation}
In particular, by \eqref{grpp.1}, \eqref{eq10.3.6}, we see that 
\begin{equation}\label{eq10.14}
|\partial_{\alpha_1}\tilde{g}(z,\alpha)| \asymp \delta F_{N+1}^2(|\zeta_-|^2).
\end{equation}
From \eqref{eq10.13},\eqref{eq10.3.5} and the Cauchy-inequalities, we obtain 
\begin{equation}\label{eq10.15}
|\partial_{\alpha_j}\tilde{g}(z,\alpha)| =\mO(\delta^2F_{N+1}^2 N), 
\quad j=2,\dots,N^2.
\end{equation}
The Cauchy-inequalities applied to \eqref{eq10.3} together with \eqref{eq10.3.5}, 
\eqref{eq10.2} yield
\begin{equation}\label{eq10.16}
\partial_{z}g(z,Q) = \partial_z g_0(z) - \delta (Q| \overline{\partial_z Z}) 
  + \frac{\mO(1)(\delta N F_{N+1}(|\zeta_-|^2))^2}{\dist(z,\partial\overline{\Sigma}_{r_0,r_1})}
\end{equation}
with
\begin{equation}\label{eq10.17}
\partial_z g_0(z)= (\partial_z \ln \zeta_- ) \frac{\zeta_-^N}{a}\left[
NF_{N+1}(\zeta_+/\zeta_-) -2(\zeta_+/\zeta_-)F_{N+1}'(\zeta_+/\zeta_-)\right].
\end{equation}
%
Here, we used as well \eqref{algrel} which implies that 
$\partial_z (\zeta_+/\zeta_-) = -(\zeta_+/\zeta_- )\partial_z \ln \zeta_-$. 
\begin{remark}
Note that in \eqref{eq10.16} 
\begin{equation}\label{eq10.17.5}
\dist(z,\partial\overline{\Sigma}_{r_0,r_1}) \geq 
\frac{\min(r_0-|\zeta_-|, |\zeta_-|-r_1)}{C}
\geq \frac{r_0-|\zeta_-|}{C}, 
\end{equation}
for some (not necessarily equal) $C\gg 1$.
\end{remark}
For $Q$ in \eqref{eq10.12.1}, we have the following result: 
\begin{lemma}\label{lem10.1} Let $Q(\alpha)\in B(0,C_1N)$ and $z\in\Omega$ 
as in \eqref{eq10.9}. Then, 
\begin{equation}\label{eq10.18}
\begin{split}
 \partial_{z}\tilde{g}(z,\alpha) = \partial_z g_0(z) - \delta \alpha_1 \partial_z |Z|^2 
  &+ \frac{\mO(1)(\delta N F_{N}(|\zeta_-|^2))^2}{\dist(z,\partial\overline{\Sigma}_{r_0,r_1})} \\
  &+\mO(\delta^2F_{N}(|\zeta_-|^2)^2N)\left| \sum_2^{N^2}\alpha_i \partial_z e_i(z) \right|,
  \end{split}
\end{equation}
\begin{equation}\label{eq10.19}
\begin{split}
 \partial_{\overline{z}}\tilde{g}(z,\alpha) =  - \delta \partial_{\overline{z}} |Z|^2\alpha_1  
  +\mO(\delta^2F_{N}(|\zeta_-|^2)^2N)\left| \alpha_1\overline{\partial_z Z} +
  \sum_2^{N^2}\alpha_i \partial_{\overline{z}} e_i(z) \right|.
  \end{split}
\end{equation}
\end{lemma}
\begin{proof}
Using \eqref{eq10.13}, one computes 
\begin{equation}
	\begin{split}
		&\partial_z \widetilde{g} \\
		&= \partial_z g_0 - \delta\alpha_1\partial_zZ\cdot \overline{Z}
		+ \partial_z(T(z,Q(\alpha,z))) \\ 
		&=\partial_z g_0 -\delta \partial_z Z\cdot\overline{Z} +(\partial_zT)(z,Q(\alpha,z))
		+ d_QT(z,Q(\alpha))\cdot \partial_zQ(\alpha,z) \\
		&=\partial_z g_0 -\delta \partial_z Z\cdot\overline{Z} +(\partial_zT)(z,Q(\alpha,z))
		+ (d_QT)(z,Q(\alpha,z))\cdot \sum_2^{N^2}\alpha_j\partial_ze_j(z),
	\end{split}
\end{equation}
where, to obtain the last equality, we used \eqref{eq10.12} and the fact that $\overline{Z}(z)$ 
is antiholomorphic in $z$. The Cauchy-inequalities together with \eqref{eq10.3.5} 
yield that 
\begin{equation}
	(\partial_zT)(z,Q(\alpha,z)) = \mO(1)
	\frac{(\delta NF_N)^2}{\mathrm{dist}(z,\partial\overline{\Sigma}_{r_0,r_1})},
\end{equation}
as well as 
\begin{equation}
	(d_QT)(z,Q(\alpha,z))\cdot \sum_2^{N^2}\alpha_j\partial_ze_j(z) 
	= \mO(\delta^2N^2F_N)\left|\sum_2^{N^2}\alpha_j\partial_ze_j(z) \right|,
\end{equation}
and we conclude \eqref{eq10.18}. Similarly, we obtain \eqref{eq10.19}.
\end{proof}
Continuing, recall that we work under assumptions \eqref{grpp.1} and 
\eqref{eq10.9} (recall as well that the last one implies \eqref{eq10.6} 
and \eqref{eq10.7}). We use \eqref{eq10.6}, \eqref{eq10.7} and apply Rouch\'e's Theorem 
to \eqref{eq10.13}, and we see that for $C_1>0$ large enough and for $|\alpha'|<C_1N$, 
the equation 
\begin{equation}\label{eq10.21}
	\tilde{g}(z,\alpha_1,\alpha') = 0
\end{equation}
has exactly one solution
\begin{equation}\label{eq10.22}
	\alpha_1=f(z,\alpha') \in D\left(0,\frac{C_1 N}{F_N(|\zeta_-|^2)}\right).
\end{equation}
Note that this yields the entire hypersurface \eqref{eq10.11.1} for 
$\Omega$ satisfying \eqref{eq10.9}, since $\tilde{g}\neq 0$ for $\alpha_1$ outside the above 
disc, which follows from \eqref{eq10.13},\eqref{eq10.3.5} and
\eqref{eq10.6}.

Moreover, $f$ satisfies 
\begin{equation}\label{eq10.23}
	f(z,\alpha')= \frac{g_0(z)}{\delta |Z|^2} + \mO(1)\delta N^2
			= \mO\left( \frac{g_0(z)}{\delta F_N(|\zeta_-|^2)^2} +\delta N^2\right).
\end{equation}
Differentiating \eqref{eq10.21} with respect to $z$ and $\overline{z}$, 
we obtain 
\begin{equation}\label{eq10.24}
 \partial_z\tilde{g} + \partial_{\alpha_1}\tilde{g}\cdot\partial_z f = 0, 
 \quad 
  \partial_{\overline{z}}\tilde{g} + \partial_{\alpha_1}\tilde{g}\cdot\partial_{\overline{z}} f = 0.
\end{equation}
Which implies that 
\begin{equation}\label{eq10.25}
 \partial_z f =  -(\partial_{\alpha_1}\tilde{g})^{-1}\partial_z\tilde{g} , \quad
 \partial_{\overline{z}} f =  -(\partial_{\alpha_1}\tilde{g})^{-1}\partial_{\overline{z}}\tilde{g}.
\end{equation}
Recall from \eqref{eq10.13} that $\tilde{g}$ is holomorphic in $\alpha_1,\dots,\alpha_{N^2}$ 
and so we see that $f$ is holomorphic in $\alpha_2,\dots,\alpha_{N^2}$. 
Applying $\partial_{\alpha_j}$, $j=2,\dots,N^2$, to \eqref{eq10.26}, we obtain 
 \begin{equation}\label{eq10.26}
	\partial_{\alpha_j}f = -(\partial_{\alpha_1}\tilde{g})^{-1}\partial_{\alpha_j}\tilde{g}, 
	\quad j=2,\dots,N^2.
 \end{equation}
Using \eqref{eq10.13} in the form 
\begin{equation}\label{eq10.27}
	\partial_{\alpha_1}\tilde{g} = - \delta |Z|^2(1+\mO(\delta F_{N+1}(|\zeta_-|^2)N)),
\end{equation}
and by Lemma \ref{lem10.1}, \eqref{eq10.25}, we obtain 
\begin{equation}\label{leq10.28}
\begin{split}
\partial_z f =& 
\frac{(1+\mO(\delta F_{N+1}(|\zeta_-|^2)N))}{\delta |Z|^2}
\bigg[ \partial_z g_0(z) - \delta (\partial_z|Z|^2)f \\
&+\frac{\mO(1)(\delta N F_{N+1}(|\zeta_-|^2))^2}{\dist(z,\partial\overline{\Sigma}_{r_0,r_1})} 
  +\mO(\delta^2F_{N+1}^2(|\zeta_-|^2)N)\left| \sum_2^{N^2}\alpha_i \partial_z e_i(z) \right|
\bigg],
\end{split}
\end{equation}
and
\begin{equation}\label{leq10.29}
\begin{split}
\partial_{\overline{z}} f = &
\frac{(1+\mO(\delta F_{N+1}(|\zeta_-|^2)N))}{\delta |Z|^2}
\bigg[ - \delta (\partial_{\overline{z}}|Z|^2)f \\
&+\mO(\delta^2F_{N+1}^2(|\zeta_-|^2)N)\left| f\overline{\partial_z Z}+
\sum_2^{N^2}\alpha_i \partial_{\overline{z}} e_i(z) \right|
\bigg].
\end{split}
\end{equation}
Furthermore, using \eqref{eq10.15} and \eqref{eq10.26}, 
we get 
\begin{equation}\label{leq10.30}
	\partial_{\alpha_j}f= \mO(1)\frac{\delta^2N F_{N+1}^2(|\zeta_-|^2)}{\delta F_N^2(|\zeta_-|^2)}
	= \mO(\delta N), \quad 
	j=2,\dots,N^2.
\end{equation}
\section{Choosing appropriate coordinates}
In the following we adopt the strategy developed in \cite[Section 5]{SjVo15}: The next 
step is to find an appropriate orthonormal basis $e_1(z),\dots,e_{N^2}(z) \in\C^{N^2}$ 
with 
\begin{equation}\label{eq10.31}
	e_1(z)= \frac{\overline{Z}(z)}{|Z(z)|},
\end{equation}
such that we obtain a good control over the terms 
$| \sum_2^{N^2}\alpha_i\partial_ze_i(z)|$, 
$|\sum_2^{N^2}\alpha_i\partial_{\overline{z}}e_i(z)|$ and such that 
the differential form $dQ_1\wedge \dots\wedge dQ_{N^2}|_{\alpha_1=f(z,\alpha')}$ 
can be expressed easily up to small errors. 
\begin{prop}\label{prop10.1}
Let $z_0\in\Sigma_{r_0-N^{-1},r_1}$. 
There exists an orthonormal basis $e_1(z),\dots,e_{N^2}(z)$ in $\C^{N^2}$ which 
depends smoothly on $z$ in a small neighbourhood of $z_0$ in 
$\C\backslash [-2\sqrt{ab},2\sqrt{ab}]$ such that 
\begin{equation*}
	\begin{split}
	&1) \quad e_1(z)= \frac{\overline{Z}(z)}{|Z(z)|}, \\
	&2) \quad \C e_1(z_0)\oplus \C e_2(z_0) =
	 \C \overline{Z}(z_0)\oplus \C \overline{\partial_zZ}(z_0), \\ 
	 & 3) \quad e_j(z) -e_j(z_0) = \mO((z_0-z)^2), ~ j=3,\dots,N^2, 
	 \text{ uniformly w.r.t. }(z,z_0).
	\end{split}
\end{equation*}
\end{prop}
\begin{proof} The proof is identical, mutatis mutandis, to the proof of 
	Proposition 5.1 in \cite{SjVo15}.
\end{proof}
As remarked after the proof of Proposition 5.1 in \cite{SjVo15}, 
we can make the following choice: 
\begin{equation}\label{eq10.32}
	e_2(z)= |f_2(z)|^{-1}f_2(z), \quad 
	f_2(z) = \overline{\partial_z Z(z)} - 
	\sum_{j\neq 2} (\overline{\partial_z Z(z)}|e_j(z))e_j(z),
\end{equation}
so that for $z=z_0$, 
\begin{equation}\label{eq10.33}
	f_2(z_0) = \overline{\partial_zZ(z_0)} - \frac{(Z(z_0)|\partial_z Z(z_0))}{|Z(z_0)|^2}
	\overline{Z(z_0)}.
\end{equation}
\begin{prop}\label{prop10.2}
For all $z\in\Sigma_1\backslash [-2\sqrt{ab},2\sqrt{ab}]$, we have 
\begin{equation}\label{eq10.34}
|\partial_z Z(z)|^2 - \frac{|(Z(z)|\partial_z Z(z))|^2}{|Z(z)|^2} 
= 2K_N(z)^2\partial_{z}\partial_{\bar{z}}\ln K_N(z),
\end{equation}
where
\begin{equation}\label{eq10.35}
	K_N(z)= \sum_{\mu=0}^{N-1}
	 \left|\frac{\zeta_-^{\mu+1} - \zeta_+^{\mu+1}}{a(\zeta_- - \zeta_+)}\right|^2
	 =\frac{1}{|a|^2}\sum_{\mu=0}^{N-1}|\zeta_-|^{2\mu}\, |F_{\mu+1}(\zeta_+/\zeta_-)|^2.
\end{equation}
\end{prop}
Before giving the proof of this proposition, let us note that by \eqref{eq10.3.4} $K_N=|Z|$.
\begin{proof}
Until further notice, we write $F_n = F_n(\zeta_+/\zeta_-)$.
First, use \eqref{grpp.14}, in the form 
\begin{equation*}
a^2 Z_{j,k} = \zeta_-^{N-j+k-1}F_{N-j+1}F_k
= \zeta_-^{\mu+\nu}F_{\mu+1}F_{\nu+1},
\end{equation*}
with $\mu = N-j$, $\nu =k-1$ and $\mu,\nu\in\{0,\dots,N-1\}$, 
to compute that 
\begin{equation*}
\begin{split}
 \frac{a^{2}}{\partial_z\ln\zeta_-}\partial_zZ_{j,k}
 = \zeta_-^{\mu+\nu} F_{\mu+1} F_{\nu+1}
 \cdot\left[ (\mu+\nu) - L_{\mu+1} - L_{\nu+1}
 \right],
 \end{split}
\end{equation*}
where $L_{n}:=\frac{2\zeta_+}{\zeta_-}\partial_t\ln F_{n}(t)|_{t=\zeta_+/\zeta_-}$. Hence, 
one obtains from the above expression and from \eqref{grpp.14} that 
\begin{equation}\label{eq10.36}
 \frac{|a|^4|(\partial_z Z|Z)|}{|\partial_z\ln\zeta_-|}= 
 \left|\sum_{\mu,\nu =0}^{N-1} |\zeta_-|^{2(\mu+\nu)}|F_{\mu+1}F_{\nu+1}|^2
 [(\mu+\nu)-L_{\mu+1} -L_{\nu+1} ]\right|.
 \end{equation}
Using \eqref{eq10.3.4} and a change of index, we obtain that 
\eqref{eq10.36} is equal to 
 \begin{equation*}
 \begin{split}
  &2\left|\sum_{\nu =0}^{N-1}|\zeta_-|^{2\nu}|F_{\nu+1}|^2\sum_{\mu =0}^{N-1} |\zeta_-|^{2\mu}
  |F_{\mu+1}|^2
 [\mu-L_{\mu+1}]\right| \\
  & =2|a|^2 |Z|\left|\sum_{\mu=0}^{N-1} |\zeta_-|^{2\mu}|F_{\mu+1}|^2
 [\mu-L_{\mu+1} ]\right|,
 \end{split}
\end{equation*}
so
 \begin{equation}\label{eq10.36.1}
   \frac{|a|^4|(\partial_z Z|Z)|}{|\partial_z\ln\zeta_-| |Z|}
   =2|a|^2\left|\sum_{\mu=0}^{N-1} |\zeta_-|^{2\mu}|F_{\mu+1}|^2
 [\mu-L_{\mu+1} ]\right|.
\end{equation}
Similarly, 
\begin{equation}\label{eq10.37}
 \frac{|a|^4 |\partial_z Z|^2}{|\partial_z\ln\zeta_-|^2}= 
 \sum_{\mu,\nu =0}^{N-1} |\zeta_-|^{2(\mu+\nu)}|F_{\mu+1}F_{\nu+1}|^2
 |(\mu+\nu)-L_{\mu+1} -L_{\nu+1} |^2.
\end{equation}
Combining \eqref{eq10.36.1}, \eqref{eq10.37}, we obtain 
\begin{equation}\label{eq10.37.1}
\begin{split}
&\frac{|a|^4}{|\partial_z\ln\zeta_-|^2}
\left(|\partial_z Z|^2-\frac{|(\partial_z Z|Z)|^2}{|Z|^2}\right) \\ 
& = \sum_{\mu,\nu =0}^{N-1} |\zeta_-|^{2(\mu+\nu)}|F_{\mu+1}F_{\nu+1}|^2
 \big[
 |(\mu+\nu)-L_{\mu+1} -L_{\nu+1} |^2 \\
 &\phantom{..........................................................}
 - 4(\mu-L_{\mu+1})(\nu-\overline{L_{\nu+1}})\big].
 \end{split}
\end{equation}
By permuting $\mu,\nu$ we get the same sum and after taking the 
average of the two expressions we may replace 
$ - 4(\mu-L_{\mu+1})(\nu-\overline{L_{\nu+1}})$ by its real part. Then, 
\begin{equation}\label{eq10.37.2}
\begin{split}
 &|(\mu+\nu)-L_{\mu+1} -L_{\nu+1} |^2 - 4\mathrm{Re}(\mu-L_{\mu+1})(\nu-\overline{L_{\nu+1}})\\
 &=|(\mu-\nu)+(L_{\nu+1}-L_{\mu+1})|^2 \\ 
 &= \left| (\mu +1) \frac{1 + t^{\mu+1}}{1 - t^{\mu+1}} - (\nu +1) \frac{1 + t^{\nu+1}}{1 - t^{\nu+1}}
      \right|^2_{t=\zeta_+/\zeta_-},
 \end{split}
\end{equation}
%
%
where we also used that by the definition of $L_{\mu}$ above and 
\eqref{eq10.2.1}
\begin{equation*}
\begin{split}
 L_{\nu+1}-L_{\mu+1} &= 
 2\frac{\zeta_+}{\zeta_-}[\partial_t \ln (1 -t^{\nu+1}) - \partial_t\ln(1-t^{\mu+1} )]_{t=\zeta_+/\zeta_-}\\
 &= \frac{2(\mu+1)t^{\mu+1}}{1-t^{\mu+1}}- \frac{2(\nu+1)t^{\nu+1}}{1-t^{\nu+1}}
 \bigg|_{t=\zeta_+/\zeta_-}.
  \end{split}
\end{equation*}
Combining this with \eqref{eq10.37.1}, we obtain
\begin{equation}\label{eq10.38}
\begin{split}
&\frac{|a|^4}{|\partial_z\ln\zeta_-|^2}
\left(|\partial_z Z|^2-\frac{|(\partial_z Z|Z)|^2}{|Z|^2}\right) \\ 
& = \sum_{\mu,\nu =0}^{N-1} |\zeta_-|^{2(\mu+\nu)}|F_{\mu+1}F_{\nu+1}|^2
 \left| (\mu +1) \frac{\zeta_-^{\mu+1} + \zeta_+^{\mu+1}}{\zeta_-^{\mu+1} - \zeta_+^{\mu+1}}
  - (\nu +1) \frac{\zeta_-^{\nu+1} + \zeta_+^{\nu+1}}{\zeta_-^{\nu+1} - \zeta_+^{\nu+1}}
      \right|^2.
 \end{split}
\end{equation}
\begin{remark}
Observe that the summands in \eqref{eq10.38} are equal to zero whenever $\mu=\nu$ 
and that the summands corresponding to the index pair $(\mu,\nu)$ is equal to the 
one corresponding to $(\nu,\mu)$. Hence, by calculating explicitly the terms for 
$(\mu,\nu)=(1,0),(0,1)$, we obtain that \eqref{eq10.38} is larger or equal than 
	\begin{equation}\label{eq10.38.1}
		2|\zeta_-|^2 |F_2 F_1|^2
		\left| 2 \frac{\zeta_-^{2} + \zeta_+^{2}}{\zeta_-^{2} - \zeta_+^{2}}
  -  \frac{\zeta_- + \zeta_+}{\zeta_- - \zeta_+}
      \right|^2.
	\end{equation}
By \eqref{eq10.2.1}, we have that $F_1(\zeta_+/\zeta_-) = 1$ and 
$F_2(\zeta_+/\zeta_-) = 1 + \zeta_+/\zeta_-$. Therefore, \eqref{eq10.38.1} 
is equal to
\begin{equation}\label{eq10.38.3}
\begin{split}
		2|\zeta_- +\zeta_+|^2
		\left| 2 \frac{\zeta_-^{2} + \zeta_+^{2}}{\zeta_-^{2} - \zeta_+^{2}}
  -  \frac{\zeta_- + \zeta_+}{\zeta_- - \zeta_+}
      \right|^2
      &=\frac{2 \left|2\zeta_-^2 +2\zeta_+^2 -\zeta_-^2-\zeta_+^2-2\zeta_-\zeta_+\right|^2}
      {|\zeta_- - \zeta_+|^2}\\
      &=2|\zeta_- -\zeta_+|^2.
      \end{split} 
	\end{equation}
Hence, 
\begin{equation}\label{eq10.39}
\left(|\partial_z Z|^2-\frac{|(\partial_z Z|Z)|^2}{|Z|^2}\right) \geq 
\frac{2|\partial_z\ln\zeta_-|^2|\zeta_- -\zeta_+|^2}{|a|^{4}} = 
\frac{2|\partial_z(\zeta_++\zeta_-)|^2 }{|a|^{4}}=
\frac{2}{|a|^6},
\end{equation}
where we used \eqref{algrel}, in particular that $\zeta_++\zeta_- = -z/a$ and 
that 
\begin{equation}\label{eq10.38.2}
\partial_z\ln\zeta_- = -\partial_z\ln\zeta_+.
\end{equation}
Thus, we conclude that for all $z\in\Sigma_1\backslash [-2\sqrt{ab},2\sqrt{ab}]$ 
the vectors $Z(z)$ and $\partial_zZ(z)$ are linearly independent. 
\end{remark}
Continuing, observe that the summands on the right hand side of \eqref{eq10.38} 
are equal to 
\begin{equation}\label{eq10.40}
 \left| (\mu +1) 
 \frac{(\zeta_-^{\mu+1} + \zeta_+^{\mu+1})(\zeta_-^{\nu+1} - \zeta_+^{\nu+1})}
 	{(\zeta_- - \zeta_+)^2}
  - (\nu +1) \frac{(\zeta_-^{\nu+1} + \zeta_+^{\nu+1})(\zeta_-^{\mu+1} - \zeta_+^{\mu+1})}
  {(\zeta_- - \zeta_+)^2}
      \right|^2.
\end{equation}
By \eqref{eq10.38.2},  
\begin{equation}\label{eq10.41}
	(\mu+1)(\zeta_-^{\mu+1} + \zeta_+^{\mu+1}) \partial_z\ln\zeta_-
	= \partial_z(\zeta_-^{\mu+1} - \zeta_+^{\mu+1}).
\end{equation}
Thus, \eqref{eq10.40} is equal to 
\begin{equation}\label{eq10.42}
\frac{|\partial_z\ln\zeta_-|^{-2}}{|\zeta_- - \zeta_+|^4}
 \left| (\zeta_-^{\nu+1} - \zeta_+^{\nu+1})\partial_z(\zeta_-^{\mu+1} - \zeta_+^{\mu+1})
  - (\zeta_-^{\mu+1} - \zeta_+^{\mu+1})\partial_z(\zeta_-^{\nu+1} - \zeta_+^{\nu+1})
      \right|^2.
\end{equation}
Writing $f_{\mu}(z)=\zeta_-^{\mu+1}(z)-\zeta_+^{\mu+1}(z)$, it follows from \eqref{eq10.38} 
and \eqref{eq10.42} that 
\begin{equation}\label{eq10.43}
\left(|\partial_z Z|^2-\frac{|(\partial_z Z|Z)|^2}{|Z|^2}\right) 
 = \frac{1}{|a|^4|\zeta_--\zeta_+|^2}\sum_{\mu,\nu =0}^{N-1} 
 \left| f_{\nu}(z)\partial_zf_{\mu}(z)-f_{\mu}(z)\partial_z f_{\nu}(z)  \right|^2.
\end{equation}
Since $f_{\mu}$ is holomorphic in $z$, we have 
$(\partial_zf_{\mu})(\overline{\partial_zf_{\mu}})=
\partial_z\partial_{\bar{z}}|f_{\mu}|^2$, and we obtain 
\begin{equation}\label{eq10.43.2}
\begin{split}
 | f_{\nu}(z)\partial_zf_{\mu}(z)-f_{\mu}(z)\partial_z f_{\nu}(z)  |^2
 = 
 |f_{\nu}(z)|^2\partial_z\partial_{\bar{z}}|f_{\mu}(z)|^2 + 
  |f_{\mu}(z)|^2\partial_z\partial_{\bar{z}}|f_{\nu}(z)|^2 \\
  -  (\partial_z |f_{\nu}(z)|^2)(\partial_{\bar{z}}|f_{\mu}(z)|^2)
  -  (\partial_z |f_{\mu}(z)|^2)(\partial_{\bar{z}}|f_{\nu}(z)|^2).
  \end{split}
\end{equation}
Using an exchange of summation index, we obtain 
from \eqref{eq10.43} and \eqref{eq10.43.2} 
\begin{equation}\label{eq10.43.1}
\begin{split}
&
\left(|\partial_z Z|^2-\frac{|(\partial_z Z|Z)|^2}{|Z|^2}\right) \\ 
& = \frac{2}{|a|^4|\zeta_--\zeta_+|^2}\sum_{\mu,\nu =0}^{N-1} 
\big[|f_{\nu}(z)|^2\partial_{z}\partial_{\bar{z}}|f_{\mu}(z)|^2 - (\partial_{z}|f_{\mu}(z)|^2)(\partial_{\bar{z}}|f_{\nu}(z)|^2)\big ] \\
& = \frac{2}{|a|^4|\zeta_--\zeta_+|^2}
\big[L_N(z) \partial_{z}\partial_{\bar{z}}L_N(z) - (\partial_{z}L_N(z))(\partial_{\bar{z}}L_N(z))\big ],
 \end{split}
\end{equation}
where $L_N(z):= \sum_{\nu =0}^{N-1} |f_{\nu}(z)|^2$, so that by \eqref{eq10.35}
$$
	K_N=\frac{L_N}{|a|^2|\zeta_- -\zeta_+|^2}
$$
Since we assumed that $z\notin [-2\sqrt{ab},2\sqrt{ab}]$, 
$\zeta_{\pm}(z)$ are holomorphic functions in $z$ and $\zeta_-\neq\zeta_+$. It follows 
that $\ln|\zeta_--\zeta_+|^2$ is harmonic, hence 
$\partial_z\partial_{\bar{z}} \ln L_N =\partial_z\partial_{\bar{z}} \ln K_N $, and 
\eqref{eq10.43} is equal to 
\begin{equation}\label{eq10.44}
2K_N^2 \partial_z\partial_{\bar{z}} \ln K_N =
2\big[K_N(z)\partial_{z}\partial_{\bar{z}}K_N(z) - \partial_z K_N(z) \partial_{\bar{z}}K_N(z)\big].
\end{equation}
\end{proof}
Next we are interested in obtaining bounds on \eqref{eq10.34}.
\begin{prop}\label{prop10.3}
Assuming \eqref{eq10.9}, we have that 
\begin{equation}\label{eq10.45}
\left(|\partial_z Z|^2-\frac{|(\partial_z Z|Z)|^2}{|Z|^2}\right) 
\asymp 
\left(F_N(|\zeta_-|^2)\right)^4.
\end{equation}
\end{prop}
\begin{proof}
For simplicity we assume that $a=1$. Recall from \eqref{eq10.9} that we 
have \eqref{eq10.8}, so $0<\sqrt{|b/a|}\leq|\zeta_-|\leq 1-1/N$, where 
we also used \eqref{signcon2} for the first two inequalities. 
\par
We write $F_{\nu+1}=F_{\nu+1}(t)$. Set $t=\zeta_+/\zeta_-$, which 
satisfies $|b/a| \leq |t|\leq 1 - 1/C$, see the remark after \eqref{eq10.4}, 
which also implies that $|F_{\nu+1}(t)|\asymp 1$.
\par
By \eqref{eq10.38}, 
\begin{equation}\label{eq10.46}
\begin{split}
&\left(|\partial_z Z|^2-\frac{|(\partial_z Z|Z)|^2}{|Z|^2}\right) \\ 
& = |\partial_z\ln\zeta_-|^2\sum_{\mu,\nu =0}^{N-1} |\zeta_-|^{2(\mu+\nu)}|F_{\mu+1}F_{\nu+1}|^2
 \left| (\mu +1) \frac{1 + t^{\mu+1}}{1- t^{\mu+1}}
  - (\nu +1) \frac{1 + t^{\nu+1}}{1 - t^{\nu+1}}
      \right|^2\\
 &\asymp \sum_{\mu,\nu =0}^{N-1} |\zeta_-|^{2(\mu+\nu)}
 \left| (\mu +1) \frac{1 + t^{\mu+1}}{1- t^{\mu+1}}
  - (\nu +1) \frac{1 + t^{\nu+1}}{1 - t^{\nu+1}}
      \right|^2
  =
  \begin{cases}
   \leq S_{2(N-1)} \\ 
   \geq S_{N-1},
  \end{cases}
 \end{split}
\end{equation}
where
\begin{equation}\label{eq10.47}
\begin{split}
&S_M = \sum_0^M |\zeta_-|^{2k}A_k, \\
&A_k = \sum_{\nu+\mu=k}
\left| (\mu +1) \frac{1 + t^{\mu+1}}{1- t^{\mu+1}}
  - (\nu +1) \frac{1 + t^{\nu+1}}{1 - t^{\nu+1}}
      \right|^2.
 \end{split}
\end{equation}
Here
$$\left|\frac{1 + t^{\mu+1}}{1- t^{\mu+1}} \right| \asymp 1, \quad  
\left|\frac{1 + t^{\nu+1}}{1- t^{\nu+1}} \right| \asymp 1,$$ 
so $A_k=\mO(k^3)$. The terms in $A_k$ with $\mu \gg \nu $ and 
$\mu \ll \nu$ are $\asymp k^2$ and there are $\asymp k$ terms 
of that kind, so $A_k\geq \frac{1}{C}k^3$, for some $C\gg 1$. 
Thus, $A_k\asymp k^3$, for $k\gg 1$.
For $k=1$, 
\begin{equation}\label{eq10.47.1}
A_1 = 2
\left| 2\frac{1 + t^{2}}{1- t^{2}}
  -  \frac{1 + t}{1 - t}
      \right|^2 = 2. 
\end{equation}
Hence, using that all $A_k \geq 0$, and that $|\zeta_-|\leq 1 - 1/N$ 
(see above), we obtain 
\begin{equation}\label{eq10.48}
S_M \asymp \sum_0^M k^3|\zeta_-|^{2k} \asymp F_M(|\zeta_-|^2)^4 .
%
\end{equation}
Here, to obtain the second estimate, we used Proposition 4.2 of \cite{SjVo15}. 
To conclude the statement of the proposition observe that $S_{2(N-1)}$ 
and $S_{N-1}$ are of the same order of magnitude, that is $F_N(|\zeta_-|^2)^4$.
\end{proof}
Continuing, recall that $F_N(\zeta_+/\zeta_-)\asymp 1$ for $z$ satisfying \eqref{eq10.9} 
and that it depends holomorphically on 
$z \in \mathring{\Sigma}_1\backslash [-2\sqrt{ab},2\sqrt{ab}]$. For simplicity, we 
sharpen assumption \eqref{eq10.9} and assume 
\begin{equation}\label{eq10.48.5}
\begin{cases}
	z\in \Sigma_{(r_0-1/N),r_1} \\
	r_0>0 \text{ satisfies  \eqref{eq10.8}}, \\
	r_1= \sqrt{|b/a|} + 1/C, ~ C\gg 1.
	\end{cases}
\end{equation}
Next, note that by the Cauchy inequalities, for $z$ satisfying \eqref{eq10.48.5}, we have 
\begin{equation}\label{eq10.49}
	|\partial_z F_N(\zeta_+/\zeta_-)| 
	\leq \mathcal{O}(1).
\end{equation}
Furthermore, $\partial_z| F_N(\zeta_+/\zeta_-)|^2 = \mathcal{O}(1)$, 
$\partial_{z}\partial_{\bar{z}} |F_N(\zeta_+/\zeta_-)|^2 = \mathcal{O}(1)$.
Using this and \cite[Proposition 4.2]{SjVo15}, we obtain for $K_N$ 
as \eqref{eq10.35} that 
\begin{equation}\label{eq10.50}
	\begin{split}
	&\partial_z K_N = \partial_z K_{\infty} 
	+\mathcal{O}\!\left( 
	\frac{N |\zeta_-|^{2N}|\partial_z\ln\zeta_-|}{1-|\zeta_-|^2}
	\right) \\
	&\partial_{\bar{z}} K_N = \partial_{\bar{z}} K_{\infty} 
	+\mathcal{O}\!\left( 
	\frac{N |\zeta_-|^{2N}|\partial_z\ln\zeta_-|}{1-|\zeta_-|^2}
	\right) \\
	&\partial_{z}\partial_{\bar{z}} K_N =\partial_{z}\partial_{\bar{z}} K_{\infty} 
	+\mathcal{O}\!\left( 
	\frac{N^2 |\zeta_-|^{2N}|\partial_z\ln\zeta_-|^2}{1-|\zeta_-|^2}
	\right),
	\end{split}
\end{equation}
where
\begin{equation}\label{eq10.51}
	\begin{split}
	&K_{\infty} \asymp \frac{1 }{1-|\zeta_-|^2} \\
	& \partial_z K_{\infty}, \partial_{\bar{z}} K_{\infty} \asymp 
	\frac{N }{1-|\zeta_-|^2} \\
	&\partial_{z}\partial_{\bar{z}} K_{\infty} \asymp
	\frac{N^2}{1-|\zeta_-|^2}.
	\end{split}
\end{equation}
Thus, by Proposition \ref{prop10.2},
\begin{equation}\label{eq10.52}
\begin{split}
|\partial_z Z(z)|^2 - &\frac{|(Z(z)|\partial_z Z(z))|^2}{|Z(z)|^2} \\
&= 2K_{\infty}(z)^2\partial_{z}\partial_{\bar{z}}\ln K_{\infty}(z) 
+\mathcal{O}\!\left( \frac{N^2 |\zeta_-|^{2N}|\partial_z\ln\zeta_-|^2}{(1-|\zeta_-|^2)^2}
	\right).
\end{split}
\end{equation}
Combining Proposition \ref{prop10.3} with \eqref{eq10.52} and 
\eqref{eq10.51} with \eqref{eq10.3.6}, we see that
\begin{equation*}
\partial_{z}\partial_{\bar{z}}\ln K_{\infty}(z)\left( 1 
+\mathcal{O}\!\left( N^2 |\zeta_-|^{2N}|\partial_z\ln\zeta_-|^2
	\right)\right)
	\asymp (F_N(|\zeta_-|^2))^2.
\end{equation*}
Since $|\zeta_-|\leq 1 - 2/N$, see \eqref{signcon2} and \eqref{eq10.48.5}, 
it then follows that 
\begin{equation}\label{eq10.52.1}
 \partial_{z}\partial_{\bar{z}}\ln K_{\infty}(z) \asymp (F_N(|\zeta_-|^2))^2.
\end{equation}
Continuing, let $e_1(z),\dots,e_{N^2}(z)$ be as in Proposition \ref{prop10.1}. 
It has been observed in \cite[Section 5]{SjVo15} that if we we assume that 
\begin{equation}\label{eq10.52.5}
	|\nabla_z e_1(z)| = \mathcal{O}(m),
\end{equation}
for some weight $m\geq 1$, then 
\begin{equation}\label{eq10.53}
	\left| 
	\sum_{3}^{N^2} \alpha_j \nabla_z e_j
	\right| \leq \mathcal{O}(m)\|\alpha\|_{\C^{N^2-2}}.
\end{equation}
In the following we shall perform the same steps as in 
\cite{SjVo15}. We present this here for the readers convenience, 
so the reader already familiar with \cite{SjVo15} may skip ahead to 
formula \eqref{eq10.59}.
\par
Next we will show that we can take the weight $m=F_N(|\zeta_-|^2)$ 
in \eqref{eq10.52.5}. 
Using, \eqref{eq10.3.6}, \eqref{eq10.31}, we have 
\begin{equation}\label{eq10.54}
	\begin{split}
	\nabla_z e_1(z) &= \frac{\nabla_z  \overline{Z}(z)}{|Z(z)|} - 
	\frac{\nabla_z |Z(z)|}{|Z(z)|^2}\overline{Z}(z) \\
	&= \frac{\nabla_z  \overline{Z}(z)}{|Z(z)|} - 
	\frac{(\nabla_z Z(z)|Z(z))+(Z(z)|\overline{\nabla}_z Z(z))}{2|Z(z)|^3}\overline{Z}(z).
	\end{split}
\end{equation}
Using \eqref{eq10.3.6} and the Cauchy inequalities, we obtain 
the estimate 
\begin{equation}\label{eq10.54.1}
	|\partial_z Z(z)| \leq \frac{F_N(|\zeta_-|^2)}{\dist(z,\partial\Sigma_{1,r_1})}
	\leq \mO(1)(F_N(|\zeta_-|^2))^2,
\end{equation}
where in the second inequality we used that, 
$\dist(z,\partial\Sigma_{1,r_1}) \geq (1-|\zeta_-|)/C$, for 
some $C\gg 1$.
\par
Since $Z$ is holomorphic, we conclude the same 
estimates for $|\nabla_z Z|$ and $|\nabla_{z}\overline{Z}|$, 
and, by using the Cauchy-inequalities, 
\begin{equation}\label{eq10.54.2}
|\partial^2_z Z| \leq \mathcal{O}(F_N^3).
\end{equation}
Using this and the fact that $K_N=|Z|$ (cf. the remark after 
Proposition \ref{prop10.2}) in \eqref{eq10.54}, we get 
\begin{equation}\label{eq10.55}
|\nabla_z e_1| = \mathcal{O}(F_N).
\end{equation}
We can therefore take $m=F_N$  in the above. Let $f_2$ be the vector as 
in \eqref{eq10.32}, so that $e_2= |f_2|^{-1}f_2$. As in the proof of 
Proposition 5.1 in \cite{SjVo15}, we let $V_0$ be the isometry from $\C^{N^2-2}$
 to $\C^{N^2}$ defined by $V_0\nu_j^0 = e_j(z_0)$, $j=3,\dots,N^2$, where 
 $\nu_3^0,\dots,\nu_{N^2}^0$ is the standard basis of $\C^{N^2-2}$. 
 Moreover, for 
 $z$ in a complex neighbourhood of $z_0$, we let $V(z)=(1-e_1(z)e_1^*(z))V_0$. 
 Setting $U(z)= V(z)(V^*(z)V(z))^{-1/2}$, we get that $e_j=U(z)\nu_j^0$, $j=3,\dots,N^2$. 
 \par
 It has been shown in \cite{SjVo15} that \eqref{eq10.52.5} implies that 
 $\| \nabla_zU(z)\| = \mathcal{O}(m)$. Thus, by \eqref{eq10.55}, we obtain 
 $\| \nabla_zU(z)\| = \mathcal{O}(F_N)$. Consider
 \begin{equation}\label{eq10.56}
 	\begin{split}
	\nabla_zf_2(z) = &\nabla_z\overline{\partial_z Z(z)} - 
	\sum_{j\neq 2}\big[(\nabla_z\overline{\partial_z Z(z)}|e_j(z))e_j(z) \\
	&+(\overline{\partial_z Z(z)}|\nabla_ze_j(z))e_j(z) 
	 + 
	(\overline{\partial_z Z(z)}|e_j(z))\nabla_z e_j(z)\big].
	\end{split}
\end{equation}
By \eqref{eq10.54.2}, we have that 
$|\nabla_z\overline{\partial_z Z(z)} | = \mathcal{O}(F_N^3)$. Moreover, 
the term for $j=1$ in the sum is of order $\mathcal{O}(F_N^3)$. It remains to 
estimate, 
\begin{equation*}
 \begin{split}
 & \mathrm{I} = \sum_{3}^{N^2} (\nabla_z\overline{\partial_z Z(z)}|e_j(z))e_j(z) \\ 
 & \mathrm{II} = \sum_{3}^{N^2}(\overline{\partial_z Z(z)}|\nabla_ze_j(z))e_j(z) \\ 
 & \mathrm{III} = \sum_{3}^{N^2} (\overline{\partial_z Z(z)}|e_j(z))\nabla_z e_j(z).
  \end{split}
\end{equation*}
Here, $|\mathrm{I}| \leq |\nabla_z\overline{\partial_z Z}(z)| = \mathcal{O}(F_N^3)$ 
and, using \eqref{eq10.53}, $|\mathrm{III}| \leq  \mathcal{O}(F_N) |\overline{\partial_z Z}(z)|
=\mathcal{O}(F_N^3) $. Moreover, 
\begin{equation*}
	\mathrm{II} = \sum _3^{N^2} (\overline{\partial_z Z}(z)|\nabla_zU(z)\nu_j^0)e_j(z)
		= \sum _3^{N^2} ((\nabla_zU(z))^*\overline{\partial_z Z}(z)|\nu_j^0)e_j(z)
\end{equation*}
which yields that $|\mathrm{II}| = |(\nabla_zU(z))^*\overline{\partial_z Z}(z)| = 
\mathcal{O}(F_N^3)$. Hence, 
 \begin{equation}\label{eq10.57}
	|\nabla_zf_2(z)| = \mathcal{O}(F_N^3).
\end{equation}
By \eqref{eq10.33}, \eqref{eq10.45}, we have that for $z=z_0$ 
 \begin{equation*}
	|f_2(z_0)|^2 = 
	 |\partial_z Z(z_0)|^2 - \frac{|(Z(z_0)|\partial_z Z(z_0))|^2}{|Z(z_0)|^2}
	 \asymp F_N(|\zeta_-|^2)^4.
\end{equation*}
Thus, for $z$ in a neighbourhood of $z_0$
\begin{equation}\label{eq10.58}
 |f_2(z)|^2\asymp F_N(|\zeta_-|^2)^4.
\end{equation}
In view of \eqref{eq10.57} we then obtain that
 $\nabla_z|f_2(z)| =\mathcal{O}(F_N^3)$. Since, 
 $e_2 = |f_2|^{-1} f_2$,
 \begin{equation*}
 |\nabla e_2(z)|= \mathcal{O}( F_N(|\zeta_-|^2)).
\end{equation*}
So, 
 \begin{equation}
\left| \sum_2^{N^2}\alpha_j\partial_z e_j\right|
\leq \mathcal{O}(F_N(|\zeta_-|^2))\|\alpha\|_{\C^{N^2-1}}
\leq \mathcal{O}(NF_N(|\zeta_-|^2)),
\end{equation}
where in the last inequality we used that $\|Q_{\omega}\| 
= \|\alpha\| \leq C_1 N$. Combining this with \eqref{leq10.28}, 
\eqref{eq10.3.6}, \eqref{eq10.23}, \eqref{eq10.3.5} and 
\eqref{eq10.17.5}, we obtain 
\begin{equation}\label{eq10.59}
	\partial_zf  = \mathcal{O}(1)\left[ 
	\frac{N|\zeta_-|^{N-1}}{\delta F_N^2} + \frac{|\zeta_-|^N}{\delta F_N}
	+\delta N^2F_N + \frac{\delta N^2}{r_0-|\zeta_-|}
	\right].
\end{equation}
Here, the first term dominates the second and the fourth term 
dominates the third, thus 
\begin{equation}\label{eq10.59.1}
	\partial_z f  = \mathcal{O}(1)\left[ 
	\frac{N|\zeta_-|^{N-1}}{\delta F_N^2}+ \delta N^3
	\right].
\end{equation}
Similarly, using \eqref{leq10.29}, 
\begin{equation}\label{eq10.59.2}
\begin{split}
	\partial_{\bar{z}}f  = &\mathcal{O}(1)\left[ 
	\frac{|\zeta_-|^{N}}{\delta F_N} + \delta N^2 F_N + N|\zeta_-|^N + \delta^2 N^3 F_{N+1}^2
	+\delta N^2F_N 
	\right] \\ 
	&=\mathcal{O}(1)\left[ \frac{|\zeta_-|^{N}}{\delta F_N} +\delta N^2 F_N   \right].
\end{split}
\end{equation}
Repeating line by line (with the obvious changes) the proof 
of Proposition 5.3 in \cite{SjVo15}, we obtain the following, 
basically, identical result:
\begin{prop}\label{prop10.4}
We express $Q$ in the canonical basis in $\C^{N^2}$ or in any other fixed orthonormal 
basis . Let $e_1(z),\dots,e_{N^2}(z)$ be an orthonormal basis in $\C^{N^2}$ depending 
smoothly on $z$, with $e_1(z)=|Z(z)|^{-1}\overline{Z}(z)$, and 
$\C e_1(z)\oplus \C e_2(z) = \C \overline{Z}(z) \oplus \C \overline{\partial_z Z}(z)$. Write 
$Q = \alpha_1\overline{Z}(z) + \sum_2^{N^2}\alpha_j e_j(z)$, and recall that the hypersurface 
\begin{equation*}
	\{(z,Q) \in \Sigma_{r_0-1/N}\backslash\Sigma_{r_1} \times B(0,C_1N); g(z,Q)=0\}, 
\end{equation*}
is given by \eqref{eq10.22} with $f$ as in \eqref{eq10.23} (see also \eqref{eq10.11.1}, 
\eqref{eq10.48.5}).  
Then, the restriction of $dQ\wedge d\overline{Q} $ to this hypersurface is given by 
\begin{equation}\label{eq10.59.5}
 \begin{split}
 &dQ\wedge d\overline{Q} = J(f) dz \wedge d\overline{z}\wedge d\alpha' \wedge d\overline{\alpha}' \\ 
 & J(f) = - \frac{|\alpha_2|^2}{|Z|^2} |(e_2|\overline{\partial_z Z})|^2 \\
 	&\phantom{J(f) = -}+\mathcal{O}(1)|\alpha_2||F_N|\left(\frac{N |\zeta_-|^{N-1}}{F_N \delta} 
				+\delta N^3 F_N + |\alpha_2| F_N^2 \delta N\right)  \\ 
	&\phantom{J(f) = -}+\mathcal{O}(1)\left(\frac{N |\zeta_-|^{N-1}}{F_N \delta} +\delta N^3 F_N 
		+ |\alpha_2| F_N^2\delta N\right)^2,
 \end{split}
\end{equation}
 where $F_N=F_N(|\zeta_-|^2)$, $\alpha' =(\alpha_2,\dots,\alpha_{N^2})$ and 
 $d\alpha' \wedge d\overline{\alpha}' = d\alpha_2 \wedge d\overline{\alpha}_2\wedge \dots \wedge d\alpha_{N^2} \wedge d\overline{\alpha}_{N^2}$.
\end{prop}
Note that the Jacobian $J(f)$ in \eqref{eq10.59.5} is invariant under any 
$z$-dependent unitary change of variables $\alpha_2,\dots,\alpha_{N^2} 
\mapsto \alpha'_2,\dots,\alpha'_{N^2} $. Therefore, to calculate $J(f)$, and 
thus $\xi$, at any given point $(z_0,\alpha_0)$ we may choose the most 
appropriate orthogonal basis $e_2(z),\dots,e_{N^2}$ in $\overline{Z}(z)^{\perp}$ 
depending smoothly on $z$.
\section{The average density}
Recall \eqref{eq10.11.2}. Using \eqref{eq10.12}, \eqref{eq10.13}, it follows by a general 
formula, obtained in Section 3 of \cite{SjVo15}, that
\begin{equation}\label{eq10.62}
 K_{\varphi}
 =
 \int \varphi(z)\xi(z) L(dz),
\end{equation}
with 
\begin{equation}\label{eq10.63}
 \xi(z) = \pi^{-N^2}\int_{{\tiny |f(z)|^2|Z(z)|^2 + |\alpha'|^2 \leq (C_1N)^2}} \e ^{-|f(z)|^2|Z(z)|^2-|\alpha'|^2}
 	J(f(z,\alpha'))L(d\alpha').
\end{equation}
where $f$ is as in \eqref{eq10.23} and $J$ is as in Proposition \ref{prop10.4}. Recall that 
we work under the hypotheses \eqref{grpp.1} and \eqref{eq10.48.5}. The latter in particular 
implies \eqref{eq10.6}, \eqref{eq10.7}. Applying these to \eqref{eq10.23} we obtain
\begin{equation}\label{eq10.64}
  |f| \leq \mathcal{O}(1)\left( \frac{g_0(z)}{\delta N F_N} + \delta N F_N \right)\frac{N}{F_N}
  \ll \frac{N}{F_N}.
\end{equation}
Now we strengthen assumptions \eqref{grpp.1}, \eqref{eq10.7} to 
\begin{equation}\label{eq10.65}
 \left( \frac{|\zeta_-|^N}{\delta N F_N} + \delta N F_N \right)
  \ll \frac{1}{N}.
\end{equation}
Then, 
\begin{equation*}
\e ^{-|f(z)|^2|Z(z)|^2} = 1 + 
\mathcal{O}(1)\left( \frac{|\zeta_-|^N}{\delta N F_N} + \delta N F_N \right)^2N^2.
\end{equation*}
Thus, using \eqref{eq10.59.5}
\begin{equation}\label{eq10.66}
	\begin{split}
	&\xi(z) = \left(1 + \mathcal{O}(1)\left( \frac{|\zeta_-|^N}{\delta N F_N} 
		 	+ \delta N F_N \right)^2N^2 \right) \cdot \\
			&\frac{|(e_2|\overline{\partial_z Z})|^2}{|Z|^2}\int_{|(f|Z|,\alpha')|\leq C_1N} 
			 |\alpha_2|^2
			\e^{-|\alpha'|^2}\pi^{-N^2}L(d\alpha')\\ 
	& + \mathcal{O}(1)\int_{|(f|Z|,\alpha')|\leq C_1N} 
	|\alpha_2||F_N|\left(\frac{N |\zeta_-|^{N-1}}{F_N \delta} 
				+\delta N^3 F_N + |\alpha_2| F_N^2 \delta N\right) 
	\e^{-|\alpha'|^2}\frac{L(d\alpha')}{\pi^{N^2}}\\
	& + \mathcal{O}(1)\int_{|(f|Z|,\alpha')|\leq C_1N} 
	\left(\frac{N |\zeta_-|^{N-1}}{F_N \delta} +\delta N^3 F_N 
		+ |\alpha_2| F_N^2\delta N\right)^2
	\e^{-|\alpha'|^2}\frac{L(d\alpha')}{\pi^{N^2}}.
	\end{split}
\end{equation}
By \eqref{eq10.64}, $|f| |Z| \ll N$. Therefore, the first integral is equal to 
\begin{equation*}
\frac{1}{\pi^2}\int |w|^2 \e^{-|w|^2}L(dw) 
+ \mathcal{O}\!\left(\e^{-\frac{N^2}{\mO(1)}}\right) 
= \frac{1}{\pi}\left(1 +  \mathcal{O}\!\left(\e^{-\frac{N^2}{\mO(1)}}\right) \right).
\end{equation*}
The sum of the other two integrals is equal to 
\begin{equation*}
\mathcal{O}(1)\left[
\left(\frac{N|\zeta_-|^{N-1}}{F_N\delta}+\delta N^3F_N\right)^2 + 
F_N\left(\frac{N|\zeta_-|^{N-1}}{F_N\delta}+\delta N^3F_N\right)
\right].
\end{equation*}
We have seen that 
\begin{equation}\label{eq10.66.5}
	\frac{|(e_2|\overline{\partial_z Z})|^2}{|Z|^2} = \mathcal{O}(F_N^2).
\end{equation}
Therefore, we obtain
\begin{equation}\label{eq10.67}
	\begin{split}
	\xi(z) = &\frac{1}{\pi}\frac{|(e_2|\overline{\partial_z Z})|^2}{|Z|^2}\\
	&+
	\mathcal{O}(1)\left[
\left(\frac{N|\zeta_-|^{N-1}}{F_N\delta}+\delta N^3F_N\right)^2 + 
F_N\left(\frac{N|\zeta_-|^{N-1}}{F_N\delta}+\delta N^3F_N\right)
\right].
	\end{split}
\end{equation}
Next, let us study the leading term in \eqref{eq10.67}. Since 
$\overline{\partial_z Z}$ belongs to the span of $e_1 = \overline{Z}/|Z|$ 
and $e_2$ for $z=z_0$, we obtain by Pythagoras' theorem that the leading term 
is equal to 
\begin{equation}\label{eq10.68}
	\frac{1}{\pi |Z|^2}\left(
	|\overline{\partial_z Z}|^2 - 
	\frac{|(\partial_z Z| Z)|^2}{|Z|^2}
	\right), \text{ for } z=z_0.
\end{equation}
By the remark after Proposition \ref{prop10.4}, this is then true 
for all $z$. 
\\
\par
%
%
Recall from the remark after Proposition \ref{prop10.2} that $K_N=|Z|$. 
Similarly to \eqref{eq10.50}, using \eqref{eq10.51} we get that 
$K_N= K_{\infty}(1 + \mathcal{O}(|\zeta_-|^{2N})$, 
where $K_{\infty} \asymp (1-|\zeta_-|^2)^{-1}$. Using this and \eqref{eq10.52}, 
we see that \eqref{eq10.67} becomes 
\begin{equation}\label{eq10.69}
	\begin{split}
	\xi(z) = &\frac{2}{\pi}\partial_{z}\partial_{\bar{z}}\ln K_{\infty}(z) 
+\mathcal{O}\!\left(N^2 |\zeta_-|^{2N}|\partial_z\ln\zeta_-|^2
	\right)\\
	&+
	\mathcal{O}(1)\left[
\left(\frac{N|\zeta_-|^{N-1}}{F_N\delta}+\delta N^3F_N\right)^2 + 
F_N\left(\frac{N|\zeta_-|^{N-1}}{F_N\delta}+\delta N^3F_N\right)
\right],
	\end{split}
\end{equation}
where by \eqref{eq10.52.1} 
\begin{equation}\label{eq10.69.5}
	\frac{2}{\pi}\partial_{z}\partial_{\bar{z}}\ln K_{\infty}(z) 
	\asymp F_N^2(|\zeta_-|^2).
\end{equation}
Thus, the error term in \eqref{eq10.69} can be written as 
\begin{equation}\label{eq10.70}
	\begin{split}
	\mathcal{O}(F_N^2)\left(\frac{N^2 |\zeta_-|^{2N}|\partial_z\ln\zeta_-|^2}{F_N^2}
	+\frac{N^2 |\zeta_-|^{2N-2}}{\delta^2 F_N^4} + \delta^2 N^6 + \frac{N |\zeta_-|^{N-1}}{\delta F_N^2} 	+ \delta N^3
	\right).
	\end{split}
\end{equation}
By \eqref{eq10.65}, we have that $(\delta F_N)^{-1} \gg N^2$. Thus, by \eqref{grpp.1} 
(which is implied by \eqref{eq10.65}), the second term in \eqref{eq10.70} is 
\begin{equation*}
 \gg \frac{N^6 |\zeta_-|^{2N-2}}{F_N^2}
\end{equation*}
which dominates the first term. Strengthening assumption \eqref{eq10.65} to 
\begin{equation}\label{eq10.71}
	\left( \frac{|\zeta_-|^{N-1} N}{\delta F_N^2} +\delta N^3\right) \ll 1,
\end{equation}
the remainder becomes
\begin{equation}\label{eq10.72}
	\mathcal{O}(F_N^2)\left(\frac{N |\zeta_-|^{N-1}}{\delta F_N^2}
 	+ \delta N^3
	\right).
\end{equation}
By \eqref{eq10.3.6}, assumption \eqref{eq10.71} is 
equivalent to 
\begin{equation}\label{eq10.72.1}
	\left( \frac{|\zeta_-|^{N-1} N}{\delta}(1-|\zeta_-|)^2 +\delta N^3\right) \ll 1.
\end{equation}
Note that for $1/C\leq r_0 \leq 1-1/N$, for some $C\gg 1$, the function 
$[0,r_0]\ni r\mapsto r^{N-1}(1-r)^2$ is increasing. Thus, unifying our 
previous assumptions, we assume that 
$z\in\Sigma_{r_0-1/N}\backslash \Sigma_{r_1}$, with $r_0$ satisfying 
$1/C\leq r_0 \leq 1-1/N$ and \eqref{eq10.72.1} with $|\zeta_-|$ replaced 
by $r_0$, and $r_1$ as in \eqref{eq10.48.5} (note that this assumption 
implies \eqref{eq10.48.5}, \eqref{grpp.1} and \eqref{eq10.72.1}). 
\par
Then, 
by \eqref{eq10.69}, \eqref{eq10.69.5}, \eqref{eq10.3.6} we conclude that 
\begin{equation}\label{eq10.73}
	\xi(z) = \frac{2}{\pi}\partial_{z}\partial_{\bar{z}}\ln K_{\infty}(z) 
	\left(1 +\mathcal{O}\!\left(\frac{N |\zeta_-|^{N-1}}{\delta}(1-|\zeta_-|)^2
 	+ \delta N^3
	\right)\right).
\end{equation}
We have proved Theorem \ref{thm1}, the main result of this paper.
\providecommand{\bysame}{\leavevmode\hbox to3em{\hrulefill}\thinspace}
\providecommand{\MR}{\relax\ifhmode\unskip\space\fi MR }
\providecommand{\MRhref}[2]{%
  \href{http://www.ams.org/mathscinet-getitem?mr=#1}{#2}
}
\providecommand{\href}[2]{#2}


\begin{thebibliography}{10}

\bibitem{BSZ00}
P.~Bleher, B.~Shiffman, and S.~Zelditch, \emph{{Universality and scaling of
  correlations between zeros on complex manifolds}}, Inventiones Mathematicae
  \textbf{142} (2000), 351--395, 10.1007/s002220000092.

\bibitem{BM}
W.~Bordeaux-Montrieux, \emph{{Loi de Weyl presque s{\^u}re et r{\'e}solvent
  pour des op{\'e}rateurs diff{\'e}rentiels non-autoadjoints, Th{\'e}se}},
  pastel.archives-ouvertes.fr/docs/00/50/12/81/PDF/manuscrit.pdf (2008).

\bibitem{BoCa16}
C.~Bordenave and M.~Capitaine, \emph{Outlier eigenvalues for deformed i.i.d random
  matrices}, Matrices. Commun. Pur. Appl. Math.. doi:10.1002/cpa.21629,
  arxiv.org/abs/1403.6001 (2016).

\bibitem{ZwChrist10}
T.J. Christiansen and M.~Zworski, \emph{{Probabilistic Weyl Laws for Quantized
  Tori}}, Communications in Mathematical Physics \textbf{299} (2010).

\bibitem{Da07}
E.~B. Davies, \emph{{Non-Self-Adjoint Operators and Pseudospectra}}, {Proc.
  Symp. Pure Math.}, vol.~76, Amer. Math. Soc., 2007.

\bibitem{DaHa09}
E.B. Davies and M.~Hager, \emph{{Perturbations of Jordan matrices}}, J. Approx.
  Theory \textbf{156} (2009), no.~1, 82--94.

\bibitem{TrEm05}
M.~Embree and L.~N. Trefethen, \emph{{Spectra and Pseudospectra: The Behavior
  of Nonnormal Matrices and Operators}}, Princeton University Press, 2005.

\bibitem{GoKh00}
I.Y. Goldsheid and B.A. Khoruzhenko, \emph{Eigenvalue curves of asymmetric
  tridiagonal random matrices}, Elec. J. of Probability. \textbf{5} (2000),
  no.~16, 1--28.

\bibitem{GuMaZe14}
A.~Guionnet, P.~Matchett Wood, and {0. Zeitouni}, \emph{{Convergence of the
  spectral measure of non-normal matrices}}, Proc.~AMS \textbf{142} (2014),
  no.~2, 667--679.

\bibitem{Ha06b}
M.~Hager, \emph{{Instabilit{\'e} Spectrale Semiclassique d{\rq}Op{\'e}rateurs
  Non-Autoadjoints II}}, Annales Henri Poincare \textbf{7} (2006), 1035--1064.

\bibitem{Ha06}
\bysame, \emph{{Instabilit{\'e} spectrale semiclassique pour des op{\'e}rateurs
  non-autoadjoints I: un mod{\`e}le}}, Annales de la facult{\'e} des sciences
  de Toulouse S{\'e}. 6 \textbf{15} (2006), no.~2, 243--280.

\bibitem{HaSj08}
M.~Hager and J.~Sj{\"o}strand, \emph{{Eigenvalue asymptotics for randomly
  perturbed non-selfadjoint operators}}, Mathematische Annalen \textbf{342}
  (2008), 177--243.

\bibitem{HaNe96}
N.~Hatano and D.R. Nelson, \emph{Localization transitions in non-hermitian
  quantum mechanics}, Physical Review Letters \textbf{77} (1996), 570--573.

\bibitem{HoKrPeVi09}
J.B. Hough, M.~Krishnapur, Y.~Peres, and B.~Vir{\'a}g, \emph{{Zeros of Gaussian
  Analytic Functions and Determinantal Point Processes}}, American Mathematical
  Society, 2009.

\bibitem{SZ03}
B.~Shiffman and S.~Zelditch, \emph{{Equilibrium distribution of zeros of random
  polynomials}}, Int. Math. Res. Not. (2003), 25--49.

\bibitem{Sj15}
J.~Sj{\"o}strand, \emph{{Non-self-adjoint differential operators, spectral
  asymptotics and random perturbations }}, Monograph in preparation,
  http://sjostrand.perso.math.cnrs.fr/.

\bibitem{SjAX1002}
\bysame, \emph{{Spectral properties of non-self-adjoint operators}}, Actes des
  Journ{\'e}es d'{\'e}.d.p. d'{\'E}vian, arxiv.org/abs/1002.4844 (2009).

\bibitem{SjVo15}
J.~Sj{\"o}strand and M.~Vogel, \emph{{Interior eigenvalue density of Jordan
  matrices with random perturbations}},  (2015), accepted for publication as
  part of a book in honour of Mikael Passare in the series Trends in
  Mathematics, Springer/Birkh{\"a}user, arxiv.org/abs/1412.2230.

\bibitem{SjVo15b}
\bysame, \emph{Large bi-diagonal matrices and random perturbations}, preprint
  arxiv.org/abs/1512.06076 (2015).

\bibitem{Sj09b}
Johannes Sj\"ostrand, \emph{Counting zeros of holomorphic functions of
  exponential growth}, Journal of pseudodifferential operators and applications
  \textbf{1} (2010), no.~1, 75--100.

\bibitem{So00}
M.~Sodin, \emph{{Zeros of Gaussian Analytic Functions and Determinantal Point
  Processes}}, Mathematical Research Letters (2000), no.~7, 371--381.

\bibitem{Vo14}
M.~Vogel, \emph{{The precise shape of the eigenvalue intensity for a class of
  non-selfadjoint operators under random perturbations}},  (2014),
  arxiv.org/abs/1401.8134.

\end{thebibliography}
\end{document}